\documentclass[twoside]{article}
\usepackage[letterpaper, left=3.5cm, right=3.5cm, top=4cm, bottom=2.5cm]{geometry}
\usepackage{graphicx}
\usepackage{amsmath,amssymb,amsthm}
\usepackage{authblk}
\usepackage[utf8]{inputenc}
\usepackage{microtype}
\usepackage{mathrsfs} 
\usepackage{enumitem}
\usepackage{calc}
\usepackage{esint}
\usepackage{fancyhdr}
\usepackage{xcolor}
\usepackage[labelfont = bf]{caption}
\usepackage{doi}
\usepackage{subfigure}
\usepackage[numbers]{natbib}
\usepackage{float}
\usepackage{stmaryrd}
\usepackage{mathtools}
\usepackage{booktabs}
\usepackage{colortbl}
\usepackage{tabulary}
\usepackage{diagbox}
\usepackage{caption}
\usepackage{cleveref}
\usepackage{commath}
\usepackage{multirow}
\usepackage{orcidlink}

\usepackage{physics}
\usepackage{amsmath}
\usepackage{tikz}
\usepackage{mathdots}
\usepackage{yhmath}
\usepackage{cancel}
\usepackage{color}
\usepackage{siunitx}
\usepackage{array}
\usepackage{multirow}
\usepackage{amssymb}
\usepackage{gensymb}
\usepackage{tabularx}
\usepackage{extarrows}
\usepackage{booktabs}
\usepackage{color,soul}
\usetikzlibrary{fadings}
\usetikzlibrary{patterns}
\usetikzlibrary{shadows.blur}
\usetikzlibrary{shapes}

\newtheorem{theorem}{Theorem}[section]
\numberwithin{equation}{section}

\newtheorem{definition}[theorem]{Definition}

\newtheorem{remark}[theorem]{Remark}
\newtheorem{lemma}[theorem]{Lemma}

\usepackage{titlesec}
\titleformat{\section}{\normalfont\scshape\centering}{\thesection.}{0.5em}{}
\titleformat*{\subsection}{\itshape}
\titleformat*{\subsubsection}{\itshape}

\providecommand{\keywords}[1]
{
	{\small\emph{Keywords:} #1}
}

\providecommand{\MSC}[1]
{
	{\small\emph{AMS MSC (2020):~~} #1}
}

\definecolor{denim}{rgb}{0.08, 0.38, 0.74}
\definecolor{byzantium}{rgb}{0.44, 0.16, 0.39} 
\definecolor{shamrockgreen}{rgb}{0.0, 0.62, 0.38} 

\usepackage{hyperref}
\hypersetup{
	colorlinks=true,
	linkcolor=denim,
	citecolor = shamrockgreen,
	filecolor=magenta, 
	urlcolor=byzantium,
}

\usepackage[textsize=small]{todonotes}
\newcommand{\setdisplayskips}[1]{%
    \setlength{\abovedisplayskip}{#1}%
    \setlength{\belowdisplayskip}{#1}%
    \setlength{\abovedisplayshortskip}{#1}%
    \setlength{\belowdisplayshortskip}{#1}%
}

\begin{document}
	\setlength{\abovedisplayskip}{5.5pt}
	\setlength{\belowdisplayskip}{5.5pt}
	\setlength{\abovedisplayshortskip}{5.5pt}
	\setlength{\belowdisplayshortskip}{5.5pt}

	\title{\vspace{-15mm}Modeling and Analysis of an Optimal Insulation Problem on Non-Smooth Domains\thanks{This work is partially supported by the Office of Naval Research (ONR) under Award NO: N00014-24-1-2147, NSF grant DMS-2408877, the Air Force Office of Scientific Research (AFOSR) under Award NO: FA9550-22-1-0248.}\vspace{-1.5mm}}
    \author[1]{Harbir Antil\,\orcidlink{0000-0002-6641-1449}%
	\thanks{Email: \url{hantil@gmu.edu}}}
    \author[2]{Alex Kaltenbach\,\orcidlink{0000-0001-6478-7963}%
	\thanks{Email: \url{kaltenbach@math.tu-berlin.de}}}
    \author[3]{Keegan L.\,A. Kirk\,\orcidlink{0000-0003-1190-6708}%
	\thanks{Email: \url{kkirk8@lsu.edu}\vspace{-1.5mm}}}
	\date{March 18, 2025\vspace{-4mm}} 
	\affil[1,3]{\small{Department of Mathematical Sciences and the Center for Mathematics and Artificial Intelligence (CMAI), George Mason University, Fairfax, VA 22030, USA.}}
	\affil[2]{\small{Institute of Mathematics, Technical University of Berlin, Stra\ss e des 17.\ Juni 136, 10623 Berlin}\vspace{-4mm}}
	\maketitle
    \vspace{-3mm}

	\pagestyle{fancy}
	\fancyhf{}
	\fancyheadoffset{0cm}
	\addtolength{\headheight}{-0.25cm}
	\renewcommand{\headrulewidth}{0pt} 
	\renewcommand{\footrulewidth}{0pt}
	\fancyhead[CO]{{Modelling of an optimal insulation problem}}
	\fancyhead[CE]{{H. Antil, A. Kaltenbach, and K. Kirk}}
	\fancyhead[R]{\thepage}
	\fancyfoot[R]{}
	
	\begin{abstract}
        \setlength{\baselineskip}{10pt}
        In this paper, we study an insulation problem that seeks the optimal distribution of a fixed amount $m>0$ of insulating material coating~an~\mbox{insulated}~boundary  $\Gamma_I\subseteq \partial\Omega$ of a thermally conducting body $\Omega\subseteq \mathbb{R}^d$, $d\in \mathbb{N}$. 
        The thickness of the thin insulating layer ${\Sigma_{I}^{\varepsilon}\subseteq \mathbb{R}^d}$ is given locally~via~$\varepsilon \mathtt{d}$, where $\mathtt{d}\colon \Gamma_{I}\to [0,+\infty)$  specifies~the~(to~be~\mbox{determined}) distribution of the  insulating material. We establish $\Gamma(L^2(\mathbb{R}^d))$-convergence of the problem (as $\varepsilon\to 0^+$). Different from the existing literature, which predominantly assumes~that~the thermally conducting body $\Omega$ has a $C^{1,1}$-boundary, we merely assume that $\Gamma_I$ is piece-wise~flat.
To over-come this lack of boundary regularity, we define the thin insulating~layer~$\Sigma_{I}^{\varepsilon}$ using a Lipschitz continuous {(globally)} transversal vector field rather than~the~outward~unit~normal~vector~field. The piece-wise flatness condition on $\Gamma_I$ is only needed to prove the~\mbox{$\liminf$-estimate}.~In~fact, for the $\limsup$-estimate it suffices that the  thermally conducting body $\Omega$ has a $C^{0,1}$-boundary.
	\end{abstract}
    \vspace{-2mm}
	
	\keywords{\hspace{-0.1mm}optimal \hspace{-0.1mm}insulation; \hspace{-0.1mm}Lipschitz \hspace{-0.1mm}domain; \hspace{-0.1mm}transversal \hspace{-0.1mm}vector \hspace{-0.1mm}field; \hspace{-0.1mm}boundary~\hspace{-0.1mm}layer;~\hspace{-0.1mm}\mbox{$\Gamma$-convergence}\vspace{-4.5mm}}
	
	\MSC{35B40; 35J25; 35Q93; 49J45; 80A19}
	
	\section{Introduction}\thispagestyle{empty}\vspace{-2mm}\enlargethispage{5mm}

    \begingroup\setdisplayskips{3.5pt}\setlength{\baselineskip}{11.5pt}
    \hspace{5mm}In the present paper, we consider the problem of determining the \emph{`best'} distribution of a given amount~of~an insulating material attached to parts of a thermally conducting body~${\Omega\hspace{-0.1em}\subseteq\hspace{-0.1em} \mathbb{R}^d}$,~${d\hspace{-0.1em}\in\hspace{-0.1em} \mathbb{N}}$, following \hspace{-0.1mm}a \hspace{-0.1mm}physical \hspace{-0.1mm}setting \hspace{-0.1mm}inspired \hspace{-0.1mm}by \hspace{-0.1mm}the \hspace{-0.1mm}work \hspace{-0.1mm}on \hspace{-0.1mm}the \hspace{-0.1mm}\emph{`thin \hspace{-0.1mm}insulation'} \hspace{-0.1mm}case~\hspace{-0.1mm}by~\hspace{-0.1mm}\mbox{{Buttazzo}}~\hspace{-0.1mm}(\textit{cf}.~\hspace{-0.1mm}\cite{Buttazzo1988}):
    given a bounded domain $\Omega\subseteq \mathbb{R}^d$, $d\in \mathbb{N}$, representing the \textit{thermally conducting body},~a~\textit{heat~source density} $f\in L^2(\Omega)$, a \textit{thin {insulating} layer} $\Sigma_\varepsilon\subseteq \mathbb{R}^d$ (\textit{i.e.}, 
 $\partial\Omega\subseteq \partial \Sigma_\varepsilon$)~of~\mbox{thickness}~$\varepsilon \mathtt{d}$, where $\varepsilon>0$ \hspace{-0.1mm}is \hspace{-0.1mm}small \hspace{-0.1mm}and \hspace{-0.1mm}$\mathtt{d}\colon \hspace{-0.175em}\partial\Omega\hspace{-0.175em}\to \hspace{-0.175em}[0,+\infty)$ \hspace{-0.1mm}a \hspace{-0.1mm}(to \hspace{-0.1mm}be \hspace{-0.1mm}determined)  \hspace{-0.1mm}\textit{distribution \hspace{-0.1mm}function}. \hspace{-0.1mm}For~\hspace{-0.1mm}${\Omega_\varepsilon\hspace{-0.175em}\coloneqq \hspace{-0.175em}\Omega\hspace{-0.15em}\cup \hspace{-0.175em}\Sigma_\varepsilon}$, we are interested in the 
    {\textit{heat loss}} functional $\smash{E}_\varepsilon^\mathtt{d}\colon \hspace{-0.05em}H^1_0(\Omega_\varepsilon)\hspace{-0.05em}\to\hspace{-0.05em} \mathbb{R}$, for~every~$v_\varepsilon\hspace{-0.05em}\in \hspace{-0.05em}H^1_0(\Omega_\varepsilon)$ defined by
    \begin{align}\label{intro:E_vareps_h}
        \smash{E}_\varepsilon^{\mathtt{d}}(v_\varepsilon)\coloneqq \tfrac{1}{2}\|\nabla v_\varepsilon\|_{\Omega}^2+\tfrac{\varepsilon}{2}\|\nabla v_\varepsilon\|_{\Sigma_\varepsilon}^2-(f,v_\varepsilon)_{\Omega}\,.
    \end{align} 
    Since the functional \eqref{intro:E_vareps_h}  is proper, strictly convex, weakly coercive, and lower semi-continuous,~the direct method in the calculus of variations yields the existence of a unique~minimizer~${u_\varepsilon^{\mathtt{d}}\in H^1_0(\Omega_\varepsilon)}$, which formally satisfies the following Euler--Lagrange equations
    \begin{align}\label{intro:ELE_Eepsh}
        \begin{aligned}
            -\Delta u_\varepsilon^{\mathtt{d}}&=f&&\quad \text{ a.e.\ in }\Omega\,,\\[-0.5mm]
            -\varepsilon\Delta u_\varepsilon^{\mathtt{d}}&=0&&\quad \text{ a.e.\ in }\Sigma_\varepsilon\,,\\[-0.5mm]
            u_\varepsilon^{\mathtt{d}}&=0&&\quad \text{ a.e.\ on }\Gamma_\varepsilon\,,\\[-0.5mm]
            \nabla (u_\varepsilon^{\mathtt{d}})^-\cdot n&=\varepsilon\nabla (u_\varepsilon^{\mathtt{d}})^+\cdot n&&\quad \text{ a.e.\ on } \partial\Omega\,, 
        \end{aligned}
    \end{align}
    where $n\colon \partial\Omega\to \mathbb{S}^{d-1}$ denotes the outward unit normal vector field to $\partial\Omega$. Moreover,  the boundary condition \eqref{intro:ELE_Eepsh}$_4$  represents a transmission condition across the boundary $\partial\Omega$,~where~$(u_\varepsilon^{\mathtt{d}})^-$~and~$(u_\varepsilon^{\mathtt{d}})^+$ denote the
    traces of $u_\varepsilon^{\mathtt{d}}$ with respect to $\Omega$ and $\Sigma_\varepsilon$, respectively.
   Physically, \eqref{intro:ELE_Eepsh}$_4$ enforces continuity of the heat flux across the interface between conducting body and insulating material, which stems from the conservation of energy, \textit{i.e.}, energy does not accumulate at the interface.\pagebreak
    \par\endgroup

    \begingroup\setdisplayskips{3.5pt}
    
    In the case $\partial\Omega\in C^{1,1}$, which is equivalent to $n\in \smash{(C^{0,1}(\partial\Omega))^d}$ (\textit{cf}.\ Remark \ref{rem:transversality}(ii)), and $\mathtt{d}\in  C^{0,1}(\partial\Omega)$ with $\mathtt{d}\ge \smash{\mathtt{d}_{\textup{min}}}$ a.e.\ on $\partial\Omega$, where $\smash{\mathtt{d}_{\textup{min}}}>0$, defining the thin {insulating}~layer~via
    \begin{align}\label{intro:sigma_eps}
        \Sigma_\varepsilon\coloneqq \smash{\big\{s+ tn(s)\mid s\in \partial\Omega\,,\; t\in [0,\varepsilon \mathtt{d}(s))\big\}}\,,
    \end{align}
    {Acerbi} and {Buttazzo} (\textit{cf}.\ \cite[Thm. II.2]{AcerbiButtazzo1986}) proved that the limit functional (as $\varepsilon\to 0^{+}$)~of~\eqref{intro:E_vareps_h} (in the sense of $\Gamma(L^2(\mathbb{R}^d))$-convergence) is given via 
    $\smash{E}^\mathtt{d}\colon H^1(\Omega)\to \mathbb{R}$, for every $v \in H^1(\Omega)$~defined~by
    \begin{align}\label{intro:E_h}
        \smash{E}^{\mathtt{d}}(v)\coloneqq \tfrac{1}{2}\|\nabla v\|_{\Omega}^2+\tfrac{1}{2}\|\smash{\mathtt{d}^{-\smash{\frac{1}{2}}}}v\|_{\partial\Omega}^2-(f,v)_{\Omega}\,.
    \end{align}
    In the functional \eqref{intro:E_h}, the first term is the internal energy of the thermally conducting~body~$\Omega$~and the \hspace{-0.1mm}third \hspace{-0.1mm}the \hspace{-0.1mm}contribution \hspace{-0.1mm}by \hspace{-0.1mm}the \hspace{-0.1mm}heat~\hspace{-0.1mm}source~\hspace{-0.1mm}density~\hspace{-0.1mm}$f$.
    \hspace{-0.25mm}The \hspace{-0.1mm}second~\hspace{-0.1mm}is~\hspace{-0.1mm}the~\hspace{-0.1mm}\mbox{\textit{`interface'}}~\hspace{-0.1mm}energy,~\hspace{-0.1mm}accoun\-ting \hspace{-0.15mm}for \hspace{-0.15mm}the \hspace{-0.15mm}interaction \hspace{-0.15mm}of \hspace{-0.15mm}the \hspace{-0.15mm}system \hspace{-0.15mm}at \hspace{-0.15mm}$\partial\Omega$ \hspace{-0.15mm}with \hspace{-0.15mm}the \hspace{-0.15mm}exterior, \hspace{-0.15mm}mediated \hspace{-0.15mm}by~\hspace{-0.15mm}the~\hspace{-0.15mm}\mbox{distribution}~\hspace{-0.15mm}\mbox{function}~\hspace{-0.15mm}$\mathtt{d}$.
    We \hspace{-0.1mm}refer \hspace{-0.1mm}to \cite{BrezisCaffarelliFriedman1980,CaffarelliFriedman1980,AcerbiButtazzo1986b,ButtazzoKohn1987,ButtazzoDalMasoMosco1989,BoutkridaMossinoMoussa1999,BoutkridaGrenonMossinoMoussa2002,MossinoVanninathan2002,PietraNitschScalaTrombetti2021,AcamporaCristoforoniNitschTrombetti2024} \hspace{-0.1mm}for \hspace{-0.1mm}related \hspace{-0.1mm}asymptotic \hspace{-0.1mm}studies \hspace{-0.1mm}and  \cite{Buttazzo1988c,Buttazzo1988,ButtazzoZeine1997,EspositoRiey2003,Buttazzo2017,HuangLiLi2022} \hspace{-0.1mm}for \hspace{-0.1mm}related \hspace{-0.1mm}analytical \hspace{-0.1mm}studies.
    \hspace{-0.1mm}Since \hspace{-0.1mm}the \hspace{-0.1mm}functional \hspace{-0.1mm}\eqref{intro:E_h} \hspace{-0.1mm}is \hspace{-0.1mm}proper, \hspace{-0.1mm}strictly \hspace{-0.1mm}convex, \hspace{-0.1mm}weakly~\hspace{-0.1mm}\mbox{coercive}, and lower semi-continuous, the direct method in the calculus of variations yields the existence~of~a unique minimizer $ u^{\mathtt{d}}\in H^1(\Omega)$,~which~formally~satisfies~the~\mbox{Euler--Lagrange}~\mbox{equations}\vspace{-0.5mm}
\begin{align}\label{intro:ELE_Eh}
        \begin{aligned}
            -\Delta u^{\mathtt{d}}&=f&&\quad \text{ a.e.\ in }\Omega\,,\\[-0.5mm]
            \mathtt{d}\nabla u^{\mathtt{d}}\cdot n+u^{\mathtt{d}}&=0&&\quad \text{ a.e.\ on }\partial\Omega\,.
        \end{aligned}
    \end{align}
    In \cite{AcerbiButtazzo1986}, the assumption $n\in \smash{(C^{0,1}(\partial\Omega))^d}$ guaranteed the existence of some $\varepsilon_0>0$~such that for every $\varepsilon\in (0,\varepsilon_0)$, the mapping ${\Phi_\varepsilon \colon D_\varepsilon\coloneqq\bigcup_{s\in \partial\Omega}{\{s\}\times [0,\varepsilon \mathtt{d}(s))}\to \Sigma_\varepsilon}$, for every $(s,t)^\top\in D_\varepsilon$~defined~by 
    \begin{align}\label{intro:phi_eps}
        \Phi_\varepsilon(s,t)\coloneqq s+ tn(s)\,,
    \end{align}
    is \hspace{-0.15mm}bi-Lipschitz \hspace{-0.15mm}continuous \hspace{-0.15mm}(\textit{i.e.}, \hspace{-0.15mm}Lipschitz \hspace{-0.15mm}continuous \hspace{-0.15mm}and \hspace{-0.15mm}bijective \hspace{-0.15mm}with \hspace{-0.15mm}Lipschitz~\hspace{-0.15mm}\mbox{continuous}~\hspace{-0.15mm}\mbox{inverse}). The latter avoids gaps (\textit{i.e.}, no insulating material is attached) or self-intersections (\textit{i.e.}, insulating material is attached twice) in the insulating boundary layer $\Sigma_\varepsilon$ (\textit{cf}.\ Figure \ref{fig:transversality}).
In applications, however,  the regularity assumption  $\partial\Omega\in C^{1,1}$ (or $n\in (C^{0,1}(\partial\Omega))^d$, respectively) is certainly too restrictive as many thermally conducting bodies in real-world applications exhibit kinks~and~edges.\linebreak For \hspace{-0.1mm}this \hspace{-0.1mm}reason,  \hspace{-0.1mm}we \hspace{-0.1mm}propose \hspace{-0.1mm}a \hspace{-0.1mm}generalization~\hspace{-0.1mm}of~\hspace{-0.1mm}the  \hspace{-0.1mm}procedure \hspace{-0.1mm}above \hspace{-0.1mm}to 
\hspace{-0.1mm}bounded~\hspace{-0.1mm}Lipschitz~\hspace{-0.1mm}\mbox{domains}.  This allows us to determine the optimal distribution of the insulating material also at kinks~and~ed\-ges (\textit{cf}.\ Figure \ref{fig:transversality} or Figure \ref{fig:thickness}). More precisely,
    our central objective is to generalize the $\Gamma(L^2(\mathbb{R}^d))$-convergence \hspace{-0.1mm}result \hspace{-0.1mm}of \hspace{-0.1mm}{Acerbi} \hspace{-0.1mm}and \hspace{-0.1mm}{Buttazzo} \hspace{-0.1mm}(\textit{cf}.\ \hspace{-0.1mm}\cite{AcerbiButtazzo1986}) \hspace{-0.1mm}to \hspace{-0.1mm}bounded \hspace{-0.1mm}Lipschitz \hspace{-0.1mm}domains~\hspace{-0.1mm}with~\hspace{-0.1mm}\mbox{piece-wise} flat boundary, which is of interest in many real-world applications~and~in~numerical~simulations. If~$\partial\Omega\in C^{0,1}$, we only have that $n\hspace{-0.1em}\in\hspace{-0.1em} (L^\infty(\partial\Omega))^d$ and, thus,~the~mapping~\eqref{intro:phi_eps}~need~not~be~\mbox{injective}. {As a consequence}, gaps or self-intersections in the thin {insulating}  layer \eqref{intro:sigma_eps}~are~not~\mbox{excluded}.\linebreak In order to avoid the latter, in the thin {insulating} layer \eqref{intro:sigma_eps}, we replace the outward~unit~\mbox{normal} vector field $n\in (L^\infty(\partial\Omega))^d$ with a Lipschitz continuous (globally) transversal~vector~field $k\in (C^{0,1}(\partial\Omega))^d$  of unit-length, \textit{i.e.}, for some $\kappa\in (0,1]$ (the \textit{transversality constant}),~we~have~that
    \begin{align*} 
            k \cdot n \ge \kappa \quad \text{ a.e.\ on }\partial\Omega\,.
    \end{align*} 
    In this case, we define the thin {insulating} layer via\vspace{-0.5mm}\enlargethispage{5mm}
    \begin{align} \label{intro:sigma_eps_new}
        \smash{\Sigma_\varepsilon\coloneqq \big\{s+ tk(s)\mid s\in \partial\Omega\,,\; t\in [0,\varepsilon \mathtt{d}(s))\big\}\,,}
    \end{align}
    which, then, enables to generalize the $\smash{\Gamma(L^2(\mathbb{R}^d))}$-convergence result of {Acerbi} and {Buttazzo} (\textit{cf}.\ \cite{AcerbiButtazzo1986}) to bounded Lipschitz domains with piece-wise flat boundary. We emphasize that the piece-wise flatness condition on the boundary is only needed to prove the $\liminf$-estimate.~In~addition, we are interested in the case, where only a part of the boundary is insulated, \textit{i.e.}, we replace $\partial\Omega$ by $\Gamma_I \subset \partial\Omega$. On the other parts, we specify either Dirichlet or Neumann~boundary~conditions.

    \textit{This paper is organized as follows:} In Sec.\ \ref{sec:preliminaries}, we introduce the relevant notation. In addition, we recall the most important definitions and results about transversal vector fields needed for the forthcoming analysis. In Sec.\ \ref{sec:modelling}, resorting to the $\Gamma$-convergence result established~in~Sec.~\ref{sec:gamma_convergence}, we perform a model reduction leading to a non-local and non-smooth~convex~minimization~problem, whose minimization enables to compute (via an explicit formula) the optimal distribution~of~the insulating material. In Sec.\ \ref{sec:tools}, we prove several auxiliary technical tools needed to establish the main result of the paper, \textit{i.e.}, the $\Gamma$-convergence result, in Sec.\ \ref{sec:gamma_convergence}.    
    \par\endgroup
    \pagebreak
	\section{Preliminaries}\label{sec:preliminaries}\vspace{-1mm}\enlargethispage{5mm}

    \subsection{Assumptions on the thermally conducting body and {insulated} boundary}\vspace{-1mm}

    \hspace{5mm}In what follows, if not otherwise specified, we assume that the thermally conducting body $\Omega\subseteq \mathbb{R}^d$, $d\in \mathbb{N}$, is a bounded Lipschitz domain with a (topological) boundary $\partial\Omega$ that~is~split~into three (relatively) open boundary parts: insulated boundary $\Gamma_I\subseteq \partial\Omega$, Dirichlet boundary~${\Gamma_D\subseteq \partial\Omega}$, and Neumann boundary $\Gamma_N\subseteq \partial\Omega$. More precisely,~we~have~that $\partial\Omega=\overline{\Gamma}_I\cup \overline{\Gamma}_D\cup \overline{\Gamma}_N$ (\textit{cf}.\ Figure~\ref{fig:domain}). In this connection, we always assume that $\Gamma_I\neq \emptyset$.\vspace{-1mm}

    \begingroup\setlength{\intextsep}{6pt}
     \begin{figure}[H]
        \captionsetup{skip=6pt}
         \centering  
         
\tikzset{every picture/.style={line width=0.75pt}} 


         \caption{\hspace{-0.15mm}A \hspace{-0.15mm}nonsmooth \hspace{-0.15mm}thermally \hspace{-0.15mm}conducting \hspace{-0.15mm}body \hspace{-0.15mm}$\Omega$ \hspace{-0.15mm}with \hspace{-0.15mm}attached \hspace{-0.15mm}insulating \hspace{-0.15mm}layer \hspace{-0.15mm}$\Sigma_{I}^{\varepsilon}$~\hspace{-0.15mm}is~\hspace{-0.15mm}\mbox{depicted}: \textit{left:} $\Gamma_I=\partial\Omega$, \textit{i.e.}, the insulating material is attached to the whole (topological) boundary $\partial\Omega$; \textit{right:} $\Gamma_I\neq \partial\Omega$, \textit{i.e.}, the insulating material is only  attached to the insulated boundary $\Gamma_I$.}\vspace{-3.5mm}
         \label{fig:domain}
     \end{figure}
    \endgroup
    
    \begingroup\setdisplayskips{3pt}
    \subsection{Notation}\vspace{-1mm}
    
    \hspace{5mm}Let $\omega\hspace{-0.1em}\subseteq \hspace{-0.1em}\mathbb{R}^n$, $n\hspace{-0.1em}\in \hspace{-0.1em}\mathbb{N}$, be a (Lebesgue) measurable set. Then, for (Lebesgue)~measurable~functions or vector fields $v,w\colon \omega\to \mathbb{R}^{\ell}$, {$\ell\in\{1,d\}$}, respectively, we employ the inner product\vspace{-0.5mm}
    \begin{align*}
        (v,w)_{\omega}\coloneqq \int_{\omega}{v\odot w\,\mathrm{d}x}\,,
    \end{align*} 
	whenever the right-hand side is well-defined, where $\smash{\odot\colon \mathbb{R}^{\ell}\times \mathbb{R}^{\ell}\to \mathbb{R}}$ either denotes~scalar~multiplication or the Euclidean inner product. If $\vert \omega\vert\in (0,+\infty)$\footnote{For a (Lebesgue) measurable set $\omega\subseteq \mathbb{R}^d$, $d\in \mathbb{N}$, we denote by $\vert \omega\vert $ its $d$-dimensional Lebesgue measure. For a $(d-1)$-dimensional submanifold $\omega\subseteq \mathbb{R}^d$, $d\in \mathbb{N}$, we denote by $\vert \omega\vert $ its $(d-1)$-dimensional~Hausdorff~measure.}, the integral mean 
    of an integrable function or vector field $v\colon \omega\to \mathbb{R}^{\ell}$, {$\ell\in\{1,d\}$}, respectively, is defined by\vspace{-0.5mm} 
    \begin{align*}
        \langle v\rangle_\omega\coloneqq  \fint_{\omega}{v\,\mathrm{d}x}\coloneqq \frac{1}{\vert \omega\vert}\int_{\omega}{v\,\mathrm{d}x}\,.
    \end{align*}
    For $p\in [1,+\infty]$, we employ standard notation for Lebesgue $\smash{L^p(\omega)}$ and Sobolev $\smash{H^{1,p}(\omega)}$ spaces, where $\omega$ should be open for  Sobolev spaces. The closure of $C_c^{\infty}(\omega)$ in $H^{1,p}(\omega)$~is~\mbox{denoted}~by~$H^{1,p}_0(\omega)$. The $L^p(\omega)$-norm is defined by\vspace{-1mm}
    \begin{align*}
        \|\cdot\|_{p,\omega}\coloneqq\begin{cases}
             (\int_\omega{\vert \cdot\vert^p\,\mathrm{d}x})^{\smash{\frac{1}{p}}}&\text{ if }p\in [1,+\infty)\,,\\
    \textup{ess\,sup}_{x\in \omega}{\vert (\cdot)(x)\vert}&\text{ if }p=+\infty\,.
        \end{cases} 
    \end{align*}
    If $p=2$, we employ the abbreviations $\smash{H^1(\omega)}\coloneqq \smash{H^{1,2}(\omega)}$, $\smash{H^1_0(\omega)}\coloneqq \smash{H^{1,2}_0(\omega)}$, and $\|\cdot\|_{\omega}\coloneqq \|\cdot\|_{2,\omega}$.\linebreak
    Moreover, we employ the same notation in the case that $\omega$ is replaced by a (relatively) open boundary part $\gamma\subseteq \partial\Omega$, in which case the Lebesgue measure $\mathrm{d}x$ is replaced by the~surface~\mbox{measure}~$\mathrm{d}s$. 

    The assumption $\Gamma_I\neq\emptyset$ ensures the validity of a
    Friedrichs' inequality (\textit{cf}.\ \cite[Ex.\ II.5.13]{Galdi}), which states that there exists a constant $c_{\mathrm{F}}>0$ such that for every ${v\in H^1(\Omega)}$,~it~holds~that
        \begin{align}\label{lem:poin_cont}
            \|v\|_{\Omega}\leq \smash{c_{\mathrm{F}}\,\big\{\|\nabla v\|_{\Omega}+\vert \langle v\rangle_{\Gamma_I}\vert\big\}}\,.\\[-6mm]\notag
        \end{align}
    \par\endgroup
    \newpage

    \begingroup\setdisplayskips{4pt}
    \subsection{Transversal vector fields}\vspace{-0.5mm}

   \hspace{5mm}In this paper, we use the following definition of a transversal vector field of~a~domain~(\textit{cf}.~\cite{HMT07}):\enlargethispage{5mm}

	\begin{definition}\label{def:transversal}
		Let $\Omega\hspace{-0.15em}\subseteq\hspace{-0.15em} \mathbb{R}^d$, $d\hspace{-0.15em}\in\hspace{-0.15em} \mathbb{N}$, be an open set of locally finite perimeter with~outward~unit~normal vector field $  n\colon \hspace{-0.15em}\partial \Omega\hspace{-0.15em}\to \hspace{-0.15em}\mathbb{S}^{d-1}$. Then, $\Omega$ has a \emph{continuous (globally) transversal~vector~field} if there exists a  vector field $k \hspace{-0.15em}\in\hspace{-0.15em} (C^0(\partial \Omega))^d$ and a constant $\kappa\hspace{-0.15em}>\hspace{-0.15em}0$~(the~\emph{transversality constant}~of~$k $)~such~that\vspace{-0.5mm}
			\begin{align}
				k \cdot  n\ge \kappa \quad\text{ a.e.\ on } \partial \Omega\,.\label{eq:transversal} 
			\end{align}
	\end{definition}

    \begin{remark}\label{rem:transversality}
        \begin{itemize}[noitemsep,topsep=2pt,leftmargin=!,labelwidth=\widthof{(ii)}]
            \item[(i)] For $\vert k\vert=1$, the condition \eqref{eq:transversal}
        in Definition \ref{def:transversal} 
        is~equivalent~to 
        \begin{align*}
            \sphericalangle(k ,  n)=\arccos(k \cdot  n)\leq \arccos(\kappa)\quad\text{ a.e.\ on }\partial\Omega\,, 
        \end{align*}
        \textit{i.e.},
        the continuous (globally) transversal vector field $k \in (C^0(\partial \Omega))^d$ varies from the outward unit normal vector field $  n\colon \partial\Omega\to \mathbb{S}^{d-1}$ up to a maximal angle of $\arccos(\kappa)$ (\textit{cf}. Figure \ref{fig:transversality}).
            \item[(ii)] {According to \cite[Thm.\ 2.19, (2.74), (2.75)]{HMT07}, if $\Omega\subseteq \mathbb{R}^d$, $d\in \mathbb{N}$,  is a non-empty, bounded \mbox{Lipschitz} domain with normal $n$, then 
            for every $\alpha\in [0,1]$, there holds $n\in (C^{0,\alpha}(\partial\Omega))^d$ if and only if $\Omega$ is a $C^{1,\alpha}$-domain.}~Therefore, if $\Omega$ is  a $C^1$-domain, the outward unit normal vector field 
            is a continuous globally transversal vector field (with transversality constant 1).
        \end{itemize}
        
    \end{remark}

    The analysis of this paper crucially relies on the following result:

    \begin{theorem}\label{thm:ex_transversal}
        Let $\Omega\subseteq \mathbb{R}^d$, $d\in \mathbb{N}$, be a non-empty, bounded Lipschitz domain. Then, $\Omega$ has a smooth (globally) transversal vector field $k\in (C^\infty(\mathbb{R}^d))^d$.
    \end{theorem}

    \begin{proof}
        See \cite[Cor.\ 2.13]{HMT07}.
    \end{proof}\vspace{-5mm}
    \par\endgroup

    \begingroup\setlength{\intextsep}{8pt}
    \begin{figure}[H]
        \captionsetup{skip=6pt}
        \centering

  
\tikzset {_msvodpya2/.code = {\pgfsetadditionalshadetransform{ \pgftransformshift{\pgfpoint{0 bp } { 0 bp }  }  \pgftransformrotate{-90 }  \pgftransformscale{2 }  }}}
\pgfdeclarehorizontalshading{_iov09x9al}{150bp}{rgb(0bp)=(0.89,0.89,0.89);
rgb(37.5bp)=(0.89,0.89,0.89);
rgb(37.5bp)=(0.82,0.82,0.82);
rgb(50bp)=(0.86,0.86,0.86);
rgb(62.5bp)=(1,1,1);
rgb(100bp)=(1,1,1)}

  
\tikzset {_v44o4ydqg/.code = {\pgfsetadditionalshadetransform{ \pgftransformshift{\pgfpoint{0 bp } { 0 bp }  }  \pgftransformrotate{-90 }  \pgftransformscale{2 }  }}}
\pgfdeclarehorizontalshading{_6it5j5m0h}{150bp}{rgb(0bp)=(0.89,0.89,0.89);
rgb(37.5bp)=(0.89,0.89,0.89);
rgb(37.5bp)=(0.82,0.82,0.82);
rgb(50bp)=(0.86,0.86,0.86);
rgb(62.5bp)=(1,1,1);
rgb(100bp)=(1,1,1)}

  
\tikzset {_bm0pqtnm7/.code = {\pgfsetadditionalshadetransform{ \pgftransformshift{\pgfpoint{0 bp } { 0 bp }  }  \pgftransformrotate{-270 }  \pgftransformscale{2 }  }}}
\pgfdeclarehorizontalshading{_h14krniqu}{150bp}{rgb(0bp)=(0.89,0.89,0.89);
rgb(37.5bp)=(0.89,0.89,0.89);
rgb(37.5bp)=(1,1,1);
rgb(49.732142857142854bp)=(0.86,0.86,0.86);
rgb(62.5bp)=(0.82,0.82,0.82);
rgb(100bp)=(0.82,0.82,0.82)}
\tikzset{every picture/.style={line width=0.75pt}} 



        \caption{A thin insulating layer  $\Sigma_{I}^{\varepsilon}$ of thickness $\varepsilon \mathtt{d}\colon  \Gamma_I\to  [0,+\infty)$ at an
         insulated boundary $\Gamma_I$ of a Lipschitz domain $\Omega$ with 
         outward unit normal  vector field ${n\colon  \partial\Omega\to  \mathbb{S}^{d-1}}$~is~\mbox{depicted}:
        \textit{top:} discontinuities of $n\colon \partial\Omega\to  \mathbb{S}^{d-1}$
        lead to gaps (\textit{i.e.}, no insulating material is applied)~or~self-intersections (\textit{i.e.}, insulating material is applied twice) in  $\widetilde{\Sigma}_{I}^{\varepsilon} \coloneqq \{s+t n(s)\mid s\in \Gamma_I\,,\;t\in (0,\varepsilon \mathtt{d}(s)]\}$;
        \textit{bottom:} gaps and self-intersections in $\Sigma_{I}^{\varepsilon}\coloneqq \{s+tk(s)\mid s\in  \Gamma_I\,,\;t\in (0,\varepsilon \mathtt{d}(s)]\}$ are avoided by replacing $n\colon\hspace{-0.1em} \partial\Omega\hspace{-0.1em}\to\hspace{-0.1em} \mathbb{S}^{d-1}$ by a unit-length continuous (globally) transversal~vector~field~${k\colon \hspace{-0.1em}\partial\Omega\hspace{-0.1em}\to\hspace{-0.1em}\mathbb{S}^{d-1}}$, which varies to $n\colon \partial\Omega\to \mathbb{S}^{d-1}$~up~to~a~maximal~angle~of~$\arccos(\kappa)$.}
        \label{fig:transversality}
    \end{figure}\endgroup\newpage

      \section{Model reduction for the thickness of the thin insulating layer}\label{sec:modelling}\vspace{-1mm}\enlargethispage{5mm}

    \hspace{5mm}Let $k\in \smash{(C^{0,1}(\partial\Omega))^d}$ be a unit-length (globally) transversal vector field of $\Omega$ with transversality constant $\kappa\in(0,1]$, whose existence is guaranteed by Theorem \ref{thm:ex_transversal}, let 
 $\varepsilon>0$~be~sufficiently small, and let $\mathtt{d}\in C^{0,1}(\Gamma_I)$ be a distribution function with $\mathtt{d}\ge \mathtt{d}_{\textup{min}}>0$. Then, we define the \textit{thin {insulating}  layer} (with respect to $k$ with thickness $\varepsilon \mathtt{d}$) $\Sigma_{I}^{\varepsilon}\subseteq \mathbb{R}^d$, the \textit{interacting insulation boundary} $\Gamma_{I}^{\varepsilon}\subseteq \partial\Sigma_{I}^{\varepsilon}$,  and the \textit{insulated body} $\Omega_{I}^{\varepsilon}\subseteq \mathbb{R}^d$, respectively, via
     \begin{align}\label{def:some_notation}
        \Sigma_{I}^{\varepsilon}&\coloneqq \big\{s+tk(s)\mid s\in \Gamma_I\,,\;t\in [0,\varepsilon \mathtt{d}(s))\big\}\,,\\
        \Gamma_{I}^{\varepsilon}&\coloneqq \big\{s+\varepsilon \mathtt{d}(s)k(s)\mid s\in \Gamma_I\big\}\,,\\
        \smash{\Omega_{I}^{\varepsilon}}&\coloneqq \smash{\Omega\cup\Sigma_{I}^{\varepsilon}}\,.
    \end{align}

    Furthermore, let $f\in \smash{L^2(\Omega)}$ be a  given \textit{heat source density}, $g\in \smash{(H^{\smash{\frac{1}{2}}}(\Gamma_N))^\ast}$ a given \textit{heat flux}, and $u_D\in H^{\smash{\frac{1}{2}}}(\Gamma_D)$ a given \textit{temperature distribution} {at the Dirichlet boundary $\Gamma_D$}.
    Then,~we~consider~the~{\textit{heat~loss}}~functional $\smash{E}_\varepsilon^\mathtt{d}\colon 
    H^1(\Omega_{I}^{\varepsilon})\to \mathbb{R}\cup\{+\infty\}$, for every $v_\varepsilon\in H^1(\Omega_{I}^{\varepsilon})$ defined by 
    \begin{align}\label{eq:Evarh}
        \smash{\smash{E}_\varepsilon^{\mathtt{d}}(v_\varepsilon)\coloneqq \tfrac{1}{2}\|\nabla v_\varepsilon\|_{\Omega}^2+\tfrac{\varepsilon}{2}\|\nabla v_\varepsilon\|_{\Sigma_{I}^{\varepsilon}}^2-(f,v_\varepsilon)_{\Omega}-\langle g,v_\varepsilon\rangle_{\smash{H^{\smash{\frac{1}{2}}}(\Gamma_N)}}+I_{\{u_D\}}^{\Gamma_D}(v_\varepsilon)+I_{\{0\}}^{\Gamma_{I}^{\varepsilon}}(v_\varepsilon)}\,,
    \end{align}
    where $I_{\{u_D\}}^{\Gamma_D}\colon H^1(\Omega)\to \mathbb{R}\cup\{+\infty\}$ and $\smash{I_{\{0\}}^{\Gamma_{I}^{\varepsilon}}}\colon H^1(\Omega_{I}^{\varepsilon})\to \mathbb{R}\cup\{+\infty\}$ denote indicator functionals, which, for every $\widehat{v}\in H^1(\Omega)$ and $\widehat{v}_\varepsilon\in H^1(\Omega_{I}^{\varepsilon})$, respectively, are defined by 
    \begin{align*}
        I_{\{u_D\}}^{\Gamma_D}(\widehat{v})\coloneqq \begin{cases}
            0&\text{ if }\widehat{v}=u_D\text{ a.e.\ on }\Gamma_D\,,\\
            +\infty&\text{ else}\,,
        \end{cases}\qquad \smash{I_{\{0\}}^{\Gamma_{I}^{\varepsilon}}}(\widehat{v}_\varepsilon)\coloneqq \begin{cases}
            0&\text{ if }\widehat{v}_\varepsilon=0\text{ a.e.\ on }\Gamma_{I}^{\varepsilon}\,,\\
            +\infty&\text{ else}\,.
        \end{cases}
    \end{align*} 
    Since the functional \eqref{eq:Evarh} is proper, strictly convex, weakly coercive, and lower semi-continuous,~the
    direct method in the calculus of variations yields the existence of a unique~minimizer~${u_\varepsilon^{\mathtt{d}}\in H^1(\Omega_{I}^{\varepsilon})}$, which formally satisfies the  Euler--Lagrange equations\vspace{-0.5mm}
    \begin{align}\label{eq:ELE_Eepsh}
        \begin{aligned}
            -\Delta u_\varepsilon^{\mathtt{d}}&=f&&\quad \text{ a.e.\ in }\Omega\,,\\[-0.5mm]
            u_\varepsilon^{\mathtt{d}}&=u_D&&\quad\text{ a.e.\ on }\Gamma_D\,,\\[-0.75mm]
            \nabla u_\varepsilon^{\mathtt{d}}\cdot n&=g&&\quad\text{ a.e.\ on }\Gamma_N\,,\\[-0.75mm]
            -\varepsilon\Delta u_\varepsilon^{\mathtt{d}}&=0&&\quad \text{ a.e.\ in }\Sigma_{I}^{\varepsilon}\,,\\[-0.75mm]
            u_\varepsilon^{\mathtt{d}}&=0&&\quad \text{ a.e.\ on }\Gamma_{I}^{\varepsilon}\,,\\[-0.75mm]
            \nabla (u^{\mathtt{d}}_\varepsilon)^-\cdot n&=\varepsilon\nabla (u^{\mathtt{d}}_\varepsilon)^+\cdot n&&\quad \text{ a.e.\ on } \Gamma_I\,,\\[-0.5mm]
        \end{aligned}
    \end{align}
    In the transmission condition \eqref{eq:ELE_Eepsh}$_6$, $\smash{(u^{\mathtt{d}}_\varepsilon)^-}$ and $\smash{(u^{\mathtt{d}}_\varepsilon)^+}$ denote the traces of $\smash{u^{\mathtt{d}}_\varepsilon}$ from $\Omega$ and $\smash{\Sigma_{I}^{\varepsilon}}$, respectively. The lateral boundary $\partial\Sigma_I^\varepsilon\setminus(\Gamma_I\cup\Gamma_I^\varepsilon)$ carries homogeneous Neumann data.

    Under the assumptions of Theorem \ref{thm:main}, passing to the limit (as $\varepsilon\to 0^{+}$)~with~a~family of trivial extensions to $L^2(\mathbb{R}^d)$ of the energy functionals $\smash{E}_\varepsilon^\mathtt{d}\colon H^1(\Omega_{I}^{\varepsilon})\to \mathbb{R}\cup\{+\infty\}$ 
    in the sense of $\Gamma(L^2(\mathbb{R}^d))$-convergence, we arrive at  the energy functional~${\smash{E}^\mathtt{d}\colon \hspace{-0.1em}H^1(\Omega)\hspace{-0.1em}\to\hspace{-0.1em} \mathbb{R}\cup\{+\infty\}}$, for every $v\in H^1(\Omega)$ defined by
    \begin{align}\label{eq:Eh}
        \smash{\smash{E}^{\mathtt{d}}(v)\coloneqq \tfrac{1}{2}\|\nabla v\|_{\Omega}^2+\tfrac{1}{2}{\|((k\cdot n)\mathtt{d})^{-\smash{\frac{1}{2}}}v\|_{\Gamma_I}^2}-(f,v)_{\Omega}-\langle g,v\rangle_{\smash{H^{\smash{\frac{1}{2}}}(\Gamma_N)}}+I_{\{u_D\}}^{\Gamma_D}(v)}\,.
    \end{align}
    In the functional \eqref{eq:Eh}, the second term, again, is the \textit{`interface'} energy,  accounting for the interaction of the system at $\Gamma_I$ with the exterior, now mediated by the scaled distribution~function~$(k\cdot n)\mathtt{d}$. \linebreak
    Since the functional \eqref{eq:Eh} is proper, strictly convex, weakly coercive, and lower semi-continuous,~the
    direct method in the calculus of variations yields the existence of a~unique~minimizer~${u^{\mathtt{d}}\in H^1(\Omega)}$, which formally satisfies the Euler--Lagrange equations\vspace{-0.5mm}
    \begin{align}\label{eq:ELE_Eh}
        \begin{aligned}
            -\Delta u^{\mathtt{d}}&=f&&\quad \text{ a.e.\ in }\Omega\,,\\[-0.5mm]
            (k\cdot n)\mathtt{d}\nabla u^{\mathtt{d}}\cdot n+u^{\mathtt{d}}&=0&&\quad \text{ a.e.\ on }\Gamma_I\,,\\[-0.75mm]
             u^{\mathtt{d}} &=u_D&&\quad \text{ a.e.\ on }\Gamma_D\,,\\[-0.75mm]
             \nabla u^{\mathtt{d}} \cdot n &=g&&\quad \text{ a.e.\ on }\Gamma_N\,,\\[-0.5mm]
        \end{aligned}
    \end{align}
    where the boundary condition on $\smash{\Gamma_I}$ in \eqref{eq:ELE_Eh}$_2$ is a Robin boundary condition. 
    
    \pagebreak
    \begingroup\setdisplayskips{4.75pt}
    We are interested in determining the non-negative distribution function $\mathtt{d}\in  L^1(\Gamma_I)$ that provides the best insulating performance, once
the total amount of insulating material~is~fixed. Note that  $\mathtt{d}\hspace{-0.15em}\in  \hspace{-0.15em}L^1(\Gamma_I)$ specifies the distribution of the insulating material in~direction~of~${k\hspace{-0.15em}\in\hspace{-0.15em} (C^0(\partial\Omega))^d}$. However, it is more natural to describe the distribution of the insulating material in  direction of  $n\in (L^\infty(\partial\Omega))^d$. The distribution of the insulating material~in~the~direction~of $n\in (L^\infty(\partial\Omega))^d$, denoted by $\widetilde{\mathtt{d}}\in L^1(\Gamma_I)$, can be computed from $\mathtt{d}\in  L^1(\Gamma_I)$~via~(\textit{cf}.~Figure~\ref{fig:thickness})\vspace{-0.5mm}
\begin{align}\label{eq:relation_h_h_tilde}
    \widetilde{\mathtt{d}}=(k\cdot n)\mathtt{d}\quad \text{ a.e.\ on }\Gamma_I\,.\\[-6mm]\notag
\end{align}
Therefore, the total amount of the insulating material should be measured in the weighted~norm $\|(k\cdot n)(\cdot)\|_{1,\Gamma_I}$ instead of $\|\cdot\|_{1,\Gamma_I}$. For bounded $\mathtt{d}$, the approximative transformation formula (\textit{cf}.\ Lemma \ref{lem:approx_trans_formula} with $v\equiv 1$) gives\vspace{-0.5mm}
\begin{align}\label{eq:measure_convergence}
    \smash{\tfrac{1}{\varepsilon}}\vert \Sigma_{I}^{\varepsilon}\vert\to \|(k\cdot n)\mathtt{d}\|_{1,\Gamma_I}\quad (\varepsilon\to 0^{+})\,.\\[-6mm]\notag
\end{align}
For this reason, we seek a distribution function $\mathtt{d}\in L^1(\Gamma_I)$ (in direction of $k$)~in~the~class\vspace{-0.5mm}
    \begin{align*}
        \mathcal{H}_{I}^m\coloneqq \big\{ \overline{\mathtt{d}}\in L^1(\Gamma_I)\mid \overline{\mathtt{d}}\ge 0\text{ a.e.\ on }\Gamma_I\,,\; \|(k\cdot n) \overline{\mathtt{d}}\|_{1,\Gamma_I}=m\big\}\,,\\[-6mm]\notag
    \end{align*}
    where \hspace{-0.1mm}$m\hspace{-0.1em}>\hspace{-0.1em}0$ \hspace{-0.1mm}is \hspace{-0.1mm}a \hspace{-0.1mm}fixed \hspace{-0.1mm}amount \hspace{-0.1mm}of \hspace{-0.1mm}the \hspace{-0.1mm}insulating \hspace{-0.1mm}material.
    With $0/0=0$ and $a/0=+\infty$ ($a>0$) in the interface integrand $\vert v\vert^2/((k\cdot n)\mathtt{d})$ of \eqref{eq:Eh}, an optimal $\mathtt{d}\in\mathcal{H}_{I}^m$ satisfies\vspace{-1mm}\enlargethispage{5mm}
    \begin{align}\label{eq:double_min}
        \min_{v\in H^1(\Omega)}{\bigl\{\smash{E}^{\mathtt{d}}(v)\bigr\}}=\min_{v\in H^1(\Omega)}{\min_{\overline{\mathtt{d}}\in \smash{\mathcal{H}_{I}^m}}{\bigl\{\smash{E}^{\overline{\mathtt{d}}}(v)\bigr\}}}=\min_{\overline{\mathtt{d}}\in \smash{\mathcal{H}_{I}^m}}{\min_{v\in H^1(\Omega)}{\bigl\{\smash{E}^{\overline{\mathtt{d}}}(v)\bigr\}}}\,.\\[-6.5mm]\notag
    \end{align}
The inner minimization problem in the middle term of \eqref{eq:double_min} defined~on~$\mathcal{H}_{I}^m$~for~fixed admissible~${v\in H^1(\Omega)}$ with $\|v\|_{1,\Gamma_I}>0$ can explicitly be solved via the formula 
    \begin{align}\label{eq:reconstruction}
        \mathtt{d}_v\coloneqq \smash{\tfrac{m}{\|v\|_{1,\Gamma_I}} \tfrac{\vert v\vert}{k\cdot n}}\in\underset{\smash{\overline{\mathtt{d}}}\in \smash{\mathcal{H}_{I}^m}}{\textup{argmin}\,}{\smash{\smash{E}^{\overline{\mathtt{d}}}(v)}}\,.\\[-7mm]\notag
    \end{align}
    {Inserting \eqref{eq:reconstruction} in \eqref{eq:double_min}, 
we arrive at a reduced problem, \textit{i.e.},} the minimization of the energy functional $I\colon H^1(\Omega)\to \mathbb{R}\cup\{+\infty\}$, for every $v\in H^1(\Omega)$ defined by\vspace{-0.5mm}
    \begin{align}
        I(v)\coloneqq  \smash{\tfrac{1}{2}}\| \nabla v\|_{\Omega}^2-(f,v)_{\Omega}+\smash{\tfrac{1}{2m}}\|v\|_{1,\Gamma_I}^2-\langle g,v\rangle_{\smash{H^{\smash{\frac{1}{2}}}(\Gamma_N)}} +I_{\{u_D\}}^{\Gamma_D}(v) \,.\label{eq:I}\\[-6mm]\notag
    \end{align}
    Since \hspace{-0.1mm}the \hspace{-0.1mm}functional \hspace{-0.1mm}\eqref{eq:I} \hspace{-0.1mm}is \hspace{-0.1mm}proper, \hspace{-0.1mm}convex, weakly coercive, \hspace{-0.1mm}and \hspace{-0.1mm}lower \hspace{-0.1mm}semi-continuous,~\hspace{-0.1mm}the~\hspace{-0.1mm}direct \hspace{-0.1mm}method \hspace{-0.1mm}in \hspace{-0.1mm}the \hspace{-0.1mm}calculus \hspace{-0.1mm}of \hspace{-0.1mm}variations \hspace{-0.1mm}yields \hspace{-0.1mm}the \hspace{-0.1mm}existence \hspace{-0.1mm}of \hspace{-0.1mm}a \hspace{-0.1mm}minimizer $u\hspace{-0.1em}\in \hspace{-0.1em}H^1(\Omega)$,~\hspace{-0.1mm}which~\hspace{-0.1mm}is~\hspace{-0.1mm}unique~\hspace{-0.1mm}if $\Gamma_D\neq \emptyset$ and $\Omega$ is connected, and formally satisfies~the~\mbox{Euler--Lagrange}~equations\vspace{-0.5mm}
    \begin{align}
    \begin{aligned}
    -\Delta u &= f&&\quad \text{ a.e.\ in }\Omega\,,\\[-0.5mm]
    -\nabla u\cdot n&\in 
    \tfrac{1}{m}(\partial\vert\cdot\vert)(u)\|u\|_{1,\Gamma_I}&&\quad\text{ a.e.\ on }\Gamma_I\,,
    \\[-0.5mm]u&=u_D&&\quad\text{ a.e.\ on }\Gamma_D\,,
    \\[-0.5mm]\nabla u\cdot n&=g&&\quad\text{ a.e.\ on }\Gamma_N\,,
    \end{aligned}\\[-6mm]\notag
\end{align}
where 
$(\partial\,\vert\cdot\vert)(t)\coloneqq\textup{sign}(t)$ if $t\neq 0$ and $(\partial\,\vert\cdot\vert)(0)\coloneqq[-1,1]$. An optimal distribution is given by \eqref{eq:reconstruction} if $\smash[b]{\|u\|_{1,\Gamma_I}}>0$. If $u=0$ a.e.\ on $\Gamma_I$, any admissible distribution is optimal.\vspace{-2.5mm}

    \setlength{\intextsep}{8pt}
    \begin{figure}[H]
        \captionsetup{skip=8pt}
        \centering

  
\tikzset {_4mttta82v/.code = {\pgfsetadditionalshadetransform{ \pgftransformshift{\pgfpoint{0 bp } { 0 bp }  }  \pgftransformrotate{-270 }  \pgftransformscale{2 }  }}}
\pgfdeclarehorizontalshading{_4slrs3ggv}{150bp}{rgb(0bp)=(1,1,1);
rgb(43.125bp)=(1,1,1);
rgb(62.5bp)=(0.89,0.89,0.89);
rgb(62.5bp)=(0.86,0.86,0.86);
rgb(62.5bp)=(0.82,0.82,0.82);
rgb(62.5bp)=(0.82,0.82,0.82);
rgb(100bp)=(0.82,0.82,0.82)}

  
\tikzset {_tvl6wo4kq/.code = {\pgfsetadditionalshadetransform{ \pgftransformshift{\pgfpoint{0 bp } { 0 bp }  }  \pgftransformrotate{-270 }  \pgftransformscale{2 }  }}}
\pgfdeclarehorizontalshading{_zkecytzyl}{150bp}{rgb(0bp)=(1,1,1);
rgb(43.125bp)=(1,1,1);
rgb(62.5bp)=(0.89,0.89,0.89);
rgb(62.5bp)=(0.86,0.86,0.86);
rgb(62.5bp)=(0.82,0.82,0.82);
rgb(62.5bp)=(0.82,0.82,0.82);
rgb(100bp)=(0.82,0.82,0.82)}

  
\tikzset {_2y8949orx/.code = {\pgfsetadditionalshadetransform{ \pgftransformshift{\pgfpoint{0 bp } { 0 bp }  }  \pgftransformrotate{-270 }  \pgftransformscale{2 }  }}}
\pgfdeclarehorizontalshading{_j54cgiqkt}{150bp}{rgb(0bp)=(1,1,1);
rgb(43.125bp)=(1,1,1);
rgb(62.5bp)=(0.89,0.89,0.89);
rgb(62.5bp)=(0.86,0.86,0.86);
rgb(62.5bp)=(0.82,0.82,0.82);
rgb(62.5bp)=(0.82,0.82,0.82);
rgb(100bp)=(0.82,0.82,0.82)}
\tikzset{every picture/.style={line width=0.75pt}} 

\vspace{-1mm}
        \caption{Sketch of relation between a distribution function $\mathtt{d}\colon \Gamma_I\to [0,+\infty)$ (in direction~of~$k$) and the associated distribution function $\widetilde{\mathtt{d}}\coloneqq (k\cdot n)\mathtt{d}\colon \Gamma_I\to [0,+\infty)$ (in direction of $n$).}
        \label{fig:thickness}
    \end{figure}\newpage
    \par\endgroup
    
    \begingroup\setdisplayskips{4pt}
    \section{Auxiliary technical tools}\label{sec:tools}\vspace{-1mm}

    \hspace{5mm}In this section, we prove auxiliary technical tools that are needed for the $\Gamma$-convergence analysis in Section \ref{sec:gamma_convergence}. To this end, for the remainder of the paper, we assume that $k\in (C^{0,1}(\partial\Omega))^d$ is a unit-length Lipschitz (globally) transversal vector field of $\Omega$ with transversality~constant~$\kappa\in(0,1]$, whose existence is ensured by Theorem \ref{thm:ex_transversal}. Let $\mathtt{d}\in L^\infty(\Gamma_I)$, with notation~\eqref{def:some_notation}. For Sobolev spaces, we use the open layer with $\mathtt{d}\in C^{0,1}(\Gamma_I)$ and $\mathtt{d}\ge\mathtt{d}_{\textup{min}}>0$.\vspace{-1.5mm}\enlargethispage{5mm}

    \subsection{Approximative transformation formula}\vspace{-0.5mm}
  
    \hspace{5mm}In this subsection, we prove an approximative transformation formula with respect~to~the~mapping $\Phi_\varepsilon \colon  D_{I}^{\varepsilon} \coloneqq   \bigcup_{s\in \Gamma_I}{\{s\}\times   [0,\varepsilon \mathtt{d}(s))} \to \Sigma_{I}^{\varepsilon}$, for every $(s,t)^\top \in  D_{I}^{\varepsilon}$~defined~by
    \begin{align*}
        \Phi_\varepsilon (s,t)\coloneqq s+tk(s)\,.
    \end{align*}
    As per the discussion in \cite[p.\ 633, 634]{HMT07}, there exists some  $\varepsilon_0>0$ such that for every $\varepsilon\in (0,\varepsilon_0)$ the mapping $\Phi_\varepsilon \colon D_{I}^{\varepsilon}\to \Sigma_{I}^{\varepsilon}$ is bi-Lipschitz continuous, so that a transformation~formula~applies.\vspace{-0.5mm}\enlargethispage{5mm}

    \begin{lemma}\label{lem:approx_trans_formula}
        For every $\varepsilon\in (0,\varepsilon_0)$ and $v\in L^1(\Sigma_{I}^{\varepsilon})$, there holds\vspace{-0.75mm}
        \begin{align*}
            \int_{\Sigma_{I}^{\varepsilon}}{v\,\mathrm{d}x}=\int_{\Gamma_I}{\int_0^{\varepsilon \mathtt{d}(s)}{v(s+t k(s))\big\{k(s)\cdot n(s)+tR_\varepsilon(s,t)\big\}\,\mathrm{d}t}\,\mathrm{d}s}\,,
        \end{align*}
        where $R_\varepsilon\in  L^\infty(D_{I}^{\varepsilon})$, $\varepsilon\in (0,\varepsilon_0)$, depends on the Lipschitz geometry of $\partial\Omega$ and on $k$ and satisfies $\smash{\sup_{\varepsilon\in (0,\varepsilon_0)}{\{\| R_\varepsilon\|_{\infty,D_{I}^{\varepsilon}}\}}<+\infty,\quad\varepsilon\|\mathtt{d}\|_{\infty,\Gamma_I}\|R_\varepsilon\|_{\infty,D_I^\varepsilon}\le\kappa/2}$.
    \end{lemma}

    \begin{proof}
        Since $\Omega$ is a bounded Lipschitz domain, there exist some $r>0$ and a finite number~${N\in \mathbb{N}}$~of affine isometric mappings $A_i\colon \hspace{-0.15em}\mathbb{R}^d\hspace{-0.15em}\to\hspace{-0.15em} \mathbb{R}^d$, $i\hspace{-0.15em}=\hspace{-0.15em}1,\ldots,N$, (\textit{i.e.}, $\mathrm{D}A_i\hspace{-0.15em}\equiv \hspace{-0.15em}O_i\hspace{-0.15em}\in\hspace{-0.15em} \mathrm{O}(d)$\footnote{$\mathrm{O}(d)\coloneqq \{ O\in\mathbb{R}^{d\times d}\mid O^{-1}=O^{\top}\}$.}~for~all~${i\hspace{-0.15em}=\hspace{-0.15em}1,\ldots,N}$) and Lipschitz mappings
        $\gamma_i\colon  B_r\coloneqq B_r^{d-1}(0)\to \mathbb{R}$, $i=1,\ldots,N$, such that $ \partial\Omega=\bigcup_{i=1}^N{A_i(\mathrm{graph}(\gamma_i))}$. We extend $\mathtt{d}$ by zero to $\partial\Omega$.
        Let \hspace{-0.1mm}$i\hspace{-0.15em}=\hspace{-0.15em}1,\ldots,N$ \hspace{-0.1mm}be \hspace{-0.1mm}arbitrary. \hspace{-0.1mm}Then, \hspace{-0.1mm}for \hspace{-0.1mm}a.e.\ \hspace{-0.1mm}$\overline{x}\hspace{-0.15em}\in\hspace{-0.15em} B_r$,~\hspace{-0.1mm}abbreviating \hspace{-0.1mm}$s_i(\overline{x})\hspace{-0.15em}\coloneqq \hspace{-0.15em}A_i(\overline{x},\gamma_{i}(\overline{x}))$,~\hspace{-0.1mm}we~\hspace{-0.1mm}have~\hspace{-0.1mm}that         
        \begin{align} \label{lem:approx_trans_formula.1}
                \smash{O_i^{\top}n(s_i(\overline{x})) =\tfrac{1}{J_{\gamma_i}(\overline{x})}(\nabla\gamma_i(\overline{x})^\top ,-1)^\top\,,\quad
            \text{where}\quad J_{\gamma_i}(\overline{x})\coloneqq (1+\vert \nabla\gamma_i(\overline{x}) \vert^2)^{\frac{1}{2}}\,.}
        \end{align}
        The mapping $F_\varepsilon^{i} \colon U_\varepsilon^{i}\coloneqq \bigcup_{\overline{x}\in B_r}{\{\overline{x}\}\times [0,\varepsilon \mathtt{d}(s_i(\overline{x})))}\to \Sigma_{I}^{\varepsilon}$, for every $(\overline{x},t)^\top \in U_\varepsilon^{i}$~defined~by
        \begin{align}\label{lem:approx_trans_formula.2}
            F_\varepsilon^{i}(\overline{x},t)\coloneqq \Phi_\varepsilon(s_i(\overline{x}),t)\,,
        \end{align} 
        is Lipschitz and, for a.e.\ $(\overline{x},t)^\top\in U_\varepsilon^{i}$, with $\mathrm{D}k$ tangential (\textit{cf}.\ \cite[Thm.\ 2.14]{AFP2000}),
        \begin{align}\label{lem:approx_trans_formula.3}
            \mathrm{D}F_\varepsilon^{i}(\overline{x},t)=
            O_i\left[\begin{array}{c|c}
                \mathrm{I}_{(d-1)\times (d-1)} & \multirow{2}{*}{$O_i^{\top}k(s_i(\overline{x}))$\!\!} \\\cmidrule(){1-1}
                \nabla \gamma_{i}(\overline{x})^\top\vphantom{X^{X^X}} & 
            \end{array}\right]+t\,\mathrm{D}k(s_i(\overline{x}))O_i\left[\begin{array}{c|c}
                \mathrm{I}_{(d-1)\times (d-1)} & \multirow{2}{*}{$0_d$\!\!} \\\cmidrule(){1-1}
                \nabla \gamma_{i}(\overline{x})^\top\vphantom{X^{X^X}} & 
            \end{array}\right].
        \end{align}
        Thus, there exists a remainder $R_{\varepsilon}^{i}\in L^\infty(U_{\varepsilon}^{i})$, depending on the Lipschitz geometry~of~$\partial\Omega$ and on $k$, with  $\sup_{\varepsilon\in (0,\varepsilon_0)}{\{\|R_{\varepsilon}^{i}\|_{\infty,U_{\varepsilon}^{i}}\}}<+\infty$  such that
        for a.e.\ $(\overline{x},t)^\top \in U_\varepsilon^{i}$,
        there holds\vspace{-0.5mm}
        \begin{align*}
            \vert \det \mathrm{D}F_\varepsilon^{i}(\overline{x},t )\vert&=
            \vert (O_i^{\top}k(s_i(\overline{x})))\cdot(\nabla\gamma_{i}(\overline{x})^\top,-1)^\top\vert +t\,R_{\varepsilon}^{i}(\overline{x},t)
            \\&=k(s_i(\overline{x}))\cdot n(s_i(\overline{x}))J_{\gamma_i}(\overline{x})+t\,R_{\varepsilon}^{i}(\overline{x},t)\,.
        \end{align*}
        Hence, if 
         $(\eta_i)_{i=1,\ldots,N}\subseteq C(\partial\Omega)$ is a partition of unity with $\textup{supp}\,\eta_i\hspace{-0.12em}\subseteq\hspace{-0.12em} s_i(B_r)\hspace{-0.1em}\subseteq\hspace{-0.1em}\partial\Omega\vphantom{C_0^\infty(\mathbb{R}^d)\Sigma_I^\varepsilon}$ and \mbox{$\sum_{i=1}^N{\eta_i}\hspace{-0.12em}=\hspace{-0.12em}1$ on $\partial\Omega$}, extended via $\eta_i(s+tk(s))\coloneqq\eta_i(s)$~for~all~$i\hspace{-0.12em}=\hspace{-0.12em}1,\ldots,N$,~then,~by~the transformation theorem, Fubini's theorem, and the definition of the surface integral,~we~arrive~at
        \begin{align*}
             &\int_{\Sigma_{I}^{\varepsilon}}{v\,\mathrm{d}x}
             = \smash{\sum_{i=1}^N}{\int_{U^{i}_\varepsilon}{(\eta_iv)\circ F_\varepsilon^{i}\vert\textrm{det}\,\mathrm{D}F_\varepsilon^{i}\vert\,\mathrm{d}\overline{x}\mathrm{d}t}}
             \\&= \smash{\sum_{i=1}^N}{\int_{B_r}{\int_0^{\varepsilon \mathtt{d}(s_i(\overline{x}))}{\hspace{-0.25em}(\eta_iv)(s_i(\overline{x})\hspace{-0.1em}+\hspace{-0.1em}tk(s_i(\overline{x}))) \big\{k(s_i(\overline{x}))\cdot n(s_i(\overline{x}))J_{\gamma_i}(\overline{x})+t\,R_{\varepsilon}^{i}(\overline{x},t)\big\}\,\mathrm{d}t}}\,\mathrm{d}\overline{x}}
             \\&=\int_{\Gamma_I}{\int_0^{\varepsilon \mathtt{d}(s)}{\hspace{-0.25em}v(s\hspace{-0.1em}+\hspace{-0.1em}tk(s))\bigg\{k(s)\cdot n(s)\hspace{-0.1em}+\hspace{-0.1em}t\,\smash{\sum_{i=1}^N}{\eta_i(s\hspace{-0.1em}+\hspace{-0.1em}tk(s))R_\varepsilon^{i}(s_i^{-1}(s),t)J_{\gamma_i}(s_i^{-1}(s))^{-1}\chi_{s_i(B_r)}}(s)\bigg\}\,\mathrm{d}t}\,\mathrm{d}s}
             \,,
        \end{align*}
        \pagebreak
        which is the claimed approximative transformation  formula.\vspace{-1mm}
    \end{proof}\newpage
    \par\endgroup

    \subsection{Lebesgue differentiation theorem with respect to vanishing boundary layers}

    \hspace{5mm}Resorting to the approximative transformation formula (\textit{cf}.\ Lemma \ref{lem:approx_trans_formula}), we are next in the position to prove a Lebesgue differentiation theorem with respect to vanishing boundary layers.\enlargethispage{5mm}
    
    \begingroup
    \setlength{\abovedisplayskip}{3.5pt}
    \setlength{\belowdisplayskip}{3.5pt}
    \setlength{\abovedisplayshortskip}{3.5pt}
    \setlength{\belowdisplayshortskip}{3.5pt}
    \begin{lemma}\label{eq:Lebesgue_boundary_limit}
        Let $a\in L^\infty(\Gamma_I)$ with $a\ge 0$ a.e.\ on $\Gamma_I$ and $v\in H^{1,p}(\Sigma_{I}^{\varepsilon_0})$, $p\in [1,+\infty)$.~Then, there holds
        \begin{align*}
            \tfrac{1}{\varepsilon}\smash{\|a^{\smash{\frac{1}{p}}}v\|_{p,\Sigma_{I}^{\varepsilon}}^p}\to \smash{\|((k\cdot n)\mathtt{d}a)^{\smash{\frac{1}{p}}}v\|_{p,\Gamma_I}^p}\quad (\varepsilon\to 0^{+})\,,
        \end{align*}
        where $a(s+tk(s))\coloneqq a(s)$ a.e.\ and $v|_{\Gamma_I}$ is the trace from $\Sigma_I^{\varepsilon_0}$.
    \end{lemma}

    \begin{proof} We proceed similarly to \cite[Lem.\ III.1]{AcerbiButtazzo1986}.
    Due to the approximative transformation formula (\textit{cf}.\ Lemma \ref{lem:approx_trans_formula}), we have that
    \begin{align}\label{eq:Lebesgue_boundary_limit.1}
    \begin{aligned} 
         \big\vert\tfrac{1}{\varepsilon}\|a^{\frac{1}{p}}v\|_{p,\Sigma_{I}^{\varepsilon}}^p-\|((k\cdot n)\mathtt{d}a)^{\frac{1}{p}}v\|_{p,\Gamma_I}^p\big\vert
         &\leq \tfrac{\|a\|_{\infty,\Gamma_I}}{\varepsilon}\big\{1+\varepsilon\|\mathtt{d}\|_{\infty,\Gamma_I}\|R_\varepsilon\|_{\infty,D_{I}^{\varepsilon}}\big\}\|F_\varepsilon\|_{1,\Gamma_I}
         \\&\quad +\|a\|_{\infty,\Gamma_I}\varepsilon\|\mathtt{d}\|_{\infty,\Gamma_I}^2\|R_\varepsilon\|_{\infty,D_{I}^{\varepsilon}}\|v\|_{p,\Gamma_I}^p\,,
         \end{aligned}
    \end{align}
    where $R_\varepsilon\in  L^\infty(D_{I}^{\varepsilon})$, $\varepsilon\in (0,\varepsilon_0)$, is as in Lemma \ref{lem:approx_trans_formula} and, for every $s\in \Gamma_I$, we define\vspace{-0.5mm}
    \begin{align}
        F_\varepsilon(s)\coloneqq \int_0^{\varepsilon \mathtt{d}(s)}{\big\vert \vert v(s+tk(s))\vert^p-\vert v(s)\vert^p\big\vert\,\mathrm{d}t}\,.\label{eq:Lebesgue_boundary_limit.2}
    \end{align} 
    For $p>1$, Taylor and Newton--Leibniz yield, for a.e.\ $s\in \Gamma_I$ and $t\in (0,\varepsilon \mathtt{d}(s)]$,
    \begin{align}\label{eq:Lebesgue_boundary_limit.3}
    \begin{aligned} 
        \big\vert\vert v(s+tk(s))\vert^p-\vert v(s)\vert^p\big\vert&=
        p\bigg\{\int_0^1{\big\{\lambda \vert v(s+tk(s))\vert+(1-\lambda) \vert v(s)\vert\big\}^{p-1}\,\mathrm{d}\lambda}\bigg\}
        \\&\quad\times\big\vert\vert v(s+tk(s))\vert-\vert v(s)\vert\big\vert
        \\[-1mm]&\leq 2^{p-1}\big\{\vert v(s+tk(s))\vert^{p-1}+\vert v(s)\vert^{p-1}\big\}\int_0^{\varepsilon \mathtt{d}(s)}{\vert \nabla v(s+\tau k(s))\vert\,\mathrm{d}\tau }\,.
        \end{aligned}\hspace{-2.5mm}
    \end{align}
    Then, using \eqref{eq:Lebesgue_boundary_limit.3} in \eqref{eq:Lebesgue_boundary_limit.2}, Hölder's inequality (with respect to $s\in \Gamma_I$), Jensen's inequality (with respect to $t\in (0,\varepsilon \mathtt{d}(s)]$), and that $k(s)\cdot n(s)+tR_\varepsilon(s,t)\ge \kappa- \varepsilon\|\mathtt{d}\|_{\infty,\Gamma_I}\|R_\varepsilon\|_{\infty,D_{I}^{\varepsilon}}\ge \kappa/2$ for a.e.\ $\smash{(s,t)^\top}\in D_I^{\varepsilon}$ together with the approximative transformation formula (\textit{cf}.\ Lemma \ref{lem:approx_trans_formula}),~we~find~that
    \begin{align}
        \|F_\varepsilon\|_{1,\Gamma_I}
        &\leq 2^{p-1}\int_{\Gamma_I}\bigg\{
           \int_0^{\varepsilon \mathtt{d}(s)}{\vert \nabla v(s+\tau k(s))\vert\,\mathrm{d}\tau }\bigg\}\notag\\&\qquad\times\bigg\{\int_0^{\varepsilon \mathtt{d}(s)}{\big\{\vert v(s+tk(s))\vert^{p-1}+\vert v(s)\vert^{p-1}\big\}\,\mathrm{d}t}\,\bigg\}\mathrm{d}s\notag
           \\&\leq  2^{p-1}\bigg(\int_{\Gamma_I}{\bigg\{
          \varepsilon \mathtt{d}(s) \fint_0^{\varepsilon \mathtt{d}(s)}{\vert \nabla v(s+t k(s))\vert\,\mathrm{d}t }\bigg\}^{\smash{p}}\mathrm{d}s}\bigg)^{\smash{\frac{1}{p}}}\notag
           \\&\quad\times \bigg(\int_{\Gamma_I}{\bigg\{\varepsilon \mathtt{d}(s)\fint_0^{\varepsilon \mathtt{d}(s)}{\big\{\vert v(s+tk(s))\vert^{p-1}+\vert v(s)\vert^{p-1}\big\}\,\mathrm{d}t}\,\bigg\}^{\smash{p'}}\mathrm{d}s}\bigg)^{\smash{\frac{1}{p'}}}\label{eq:Lebesgue_boundary_limit.4}
           \\&\leq  2^{p-1+1/p}\bigg(\int_{\Gamma_I}{
          (\varepsilon \mathtt{d}(s))^{p-1} \int_0^{\varepsilon \mathtt{d}(s)}{\vert \nabla v(s+t k(s))\vert^p\tfrac{k(s)\cdot n(s)+tR_\varepsilon(s,t)}{\kappa- \varepsilon\|\mathtt{d}\|_{\infty,\Gamma_I}\|R_\varepsilon\|_{\infty,D_{I}^{\varepsilon}}} \,\mathrm{d}t }\,\mathrm{d}s}\bigg)^{\smash{\frac{1}{p}}}\notag
           \\&\quad\times \bigg(\int_{\Gamma_I}{(\varepsilon \mathtt{d}(s))^{\frac{1}{p-1}}\int_0^{\varepsilon \mathtt{d}(s)}{\big\{\vert v(s+tk(s))\vert^p
           \tfrac{k(s)\cdot n(s)+tR_\varepsilon(s,t)}{\kappa- \varepsilon\|\mathtt{d}\|_{\infty,\Gamma_I}\|R_\varepsilon\|_{\infty,D_{I}^{\varepsilon}}}+\vert v(s)\vert^p\big\}\,\mathrm{d}t}\,\mathrm{d}s}\bigg)^{\smash{\frac{1}{p'}}}\notag
            \\&\leq  2^{p-1+1/p}\smash{(\tfrac{\|\mathtt{d}\|_{\infty,\Gamma_I}\varepsilon}{\kappa- \varepsilon\|\mathtt{d}\|_{\infty,\Gamma_I}\|R_\varepsilon\|_{\infty,D_{I}^{\varepsilon}}}})^{\frac{1}{p'}}
           \|\nabla v\|_{p,\Sigma_{I}^{\varepsilon}}\notag \\[0.5mm]&\quad\times\smash{(\tfrac{\|\mathtt{d}\|_{\infty,\Gamma_I}\varepsilon}{\kappa- \varepsilon\|\mathtt{d}\|_{\infty,\Gamma_I}\|R_\varepsilon\|_{\infty,D_{I}^{\varepsilon}}})^{\smash{\frac{1}{p}}}}
            \big\{\smash[b]{\|v\|_{p,\Sigma_{I}^{\varepsilon}}^p}\vphantom{\|v\|_{p,\Sigma_{I}^{\varepsilon}}}+\|\mathtt{d}\|_{\infty,\Gamma_I}\varepsilon\smash[b]{\|v\|_{p,\Gamma_I}^p}\vphantom{\|v\|_{p,\Gamma_I}}\big\}^{\smash{\frac{1}{p'}}}\,.\notag
    \end{align}
    Eventually, using that $\frac{1}{p}+\frac{1}{p'}=1$, we conclude that\vspace{-0.5mm}
    \begin{align*}
        \smash{\tfrac{1}{\varepsilon}\|F_\varepsilon\|_{1,\Gamma_I}}&\leq 2^{p+1}\smash{\tfrac{\|\mathtt{d}\|_{\infty,\Gamma_I}}{\kappa}}\|\nabla v\|_{p,\Sigma_{I}^{\varepsilon}}\big\{\smash[b]{\|v\|_{p,\Sigma_{I}^{\varepsilon}}^p}\vphantom{\|v\|_{p,\Sigma_{I}^{\varepsilon}}}+\|\mathtt{d}\|_{\infty,\Gamma_I}\varepsilon\smash[b]{\|v\|_{p,\Gamma_I}^p}\vphantom{\|v\|_{p,\Gamma_I}}\big\}^{\smash{\frac{1}{p'}}}
        \to 0\quad (\varepsilon\to 0^{+})\,,
    \end{align*}
    For $p=1$, $\|F_\varepsilon\|_{1,\Gamma_I}\le C\varepsilon\|\nabla v\|_{1,\Sigma_I^\varepsilon}=o(\varepsilon)$ by Newton--Leibniz. Now use \eqref{eq:Lebesgue_boundary_limit.1}.
    \end{proof}
    \endgroup

    \subsection{Point-wise Poincar\'e inequality in thin boundary layers}\vspace{-0.5mm}
    
    \hspace{5mm}In the forthcoming analysis, we will frequently resort to the following point-wise Poincar\'e inequality for Sobolev functions  defined in the thin {insulating} layer
    $\Sigma_{I}^{\varepsilon}$ and vanishing on the interacting insulation boundary $\Gamma_I^\varepsilon$.\vspace{-0.5mm}\enlargethispage{5mm}

    \begin{lemma}\label{lem:poincare}
        Let  $v_\varepsilon\in H^1(\Sigma_{I}^{\varepsilon})$ with $v_\varepsilon=0$ a.e.\ on $\Gamma_I^\varepsilon$. Then, for a.e.\ $s\in \Gamma_I$ and $t\in [0,\varepsilon \mathtt{d}(s)]$, there holds
        \begin{align*}
            \vert v_\varepsilon(s+tk(s))\vert^2\leq (\varepsilon \mathtt{d}(s)-t)
            \int_t^{\varepsilon \mathtt{d}(s)}{\vert \nabla v_\varepsilon (s+\lambda k(s))\vert^2\,\mathrm{d}\lambda}\,.
        \end{align*}
    \end{lemma}

    \begin{proof}
        Using the Newton--Leibniz formula and Jensen's inequality,
        for a.e.\ $s\hspace{-0.05em}\in \hspace{-0.05em}\Gamma_I$~and~${t\hspace{-0.05em}\in\hspace{-0.05em} [0,\varepsilon \mathtt{d}(s)]}$, using that $v_\varepsilon (s+\varepsilon \mathtt{d}(s) k(s))=0$ (since  $s+\varepsilon \mathtt{d}(s) k(s)\in \Gamma_I^\varepsilon$), we find that
        \begin{align*}
            \vert v_\varepsilon (s+t k(s))\vert^2
            &
            =\left\vert\int_t^{\varepsilon \mathtt{d}(s)}{\nabla v_\varepsilon (s+\lambda k(s))\cdot k(s)\,\mathrm{d}\lambda}\right\vert^2
            \\&\leq (\varepsilon \mathtt{d}(s)-t)
            \int_t^{\varepsilon \mathtt{d}(s)}{\vert \nabla v_\varepsilon (s+\lambda k(s))\cdot k(s)\vert^2\,\mathrm{d}\lambda}\,,
        \end{align*}
        which, using that $\vert k\vert=1$ a.e.\ on $\Gamma_I$, yields the claimed point-wise Poincar\'e inequality.
    \end{proof}

    \subsection{Equi-coercivity}\vspace{-0.5mm}

    \hspace{5mm}The family of functionals $\smash{E}_\varepsilon^\mathtt{d}\colon H^1(\Omega_{I}^{\varepsilon})\to\mathbb{R}\cup\{+\infty\}$, $\varepsilon\in (0,\varepsilon_0)$, is equi-coercive.\vspace{-0.5mm}

    \begin{lemma}\label{lem:equicoercive}
    For a sequence $v_\varepsilon\in H^1(\Omega_{I}^{\varepsilon})$, $\varepsilon\in (0,\varepsilon_0)$, from 
   \begin{align*}
      \sup_{\varepsilon\in (0,\varepsilon_0)}{\big\{\smash{E}_\varepsilon^{\mathtt{d}}(v_\varepsilon)\big\}}<+\infty  \,,
   \end{align*}
   it~follows~that
        \begin{align*}
         \sup_{\varepsilon\in (0,\varepsilon_0)}{\big\{\|v_\varepsilon\|_{\Omega}^2+\|\nabla v_\varepsilon\|_{\Omega}^2+\tfrac{1}{\varepsilon}\|v_\varepsilon\|_{\Sigma_{I}^{\varepsilon}}^2+\varepsilon\|\nabla v_\varepsilon\|_{\Sigma_{I}^{\varepsilon}}^2\big\}}<+\infty \,.
        \end{align*}
    \end{lemma}
 
    \begin{proof} We proceed similarly to \cite[Thm III.3]{AcerbiButtazzo1986}. 
        Let $c_E\coloneqq \sup_{\varepsilon\in (0,\varepsilon_0)}{\{\smash{E}_\varepsilon^{\mathtt{d}}(v_\varepsilon)\}}$. To~begin~with,  we observe that $v_\varepsilon=0$ a.e.\ on $\Gamma_I^{\varepsilon}$, $v_\varepsilon=u_D$ a.e.\ on $\Gamma_D$, and, by the weighted Young's inequality, for every $\delta>0$, that
        \begin{align}\label{lem:equicoercive.1}
            \tfrac{1}{2}\|\nabla v_\varepsilon\|_{\Omega}^2+\tfrac{\varepsilon}{2}\smash[b]{\|\nabla v_\varepsilon\|_{\Sigma_{I}^{\varepsilon}}^2}\vphantom{\|\nabla v_\varepsilon\|_{\Sigma_{I}^{\varepsilon}}}\leq c_E+\tfrac{1}{2\delta}\big\{\|f\|_{\Omega}^2+\| g\|_{\smash{(H^{\smash{\frac{1}{2}}}(\Gamma_N))^*}}^2\big\}+\tfrac{\delta}{2}\big\{\|v_\varepsilon\|_{\Omega}^2+\| v_\varepsilon\|_{\smash{H^{\smash{\frac{1}{2}}}(\Gamma_N)}}^2\big\}\,.
        \end{align}
        Using 
        \hspace{-0.15mm}the \hspace{-0.15mm}approximative \hspace{-0.15mm}transformation \hspace{-0.15mm}formula \hspace{-0.15mm}(\textit{cf}.\ \hspace{-0.15mm}Lemma \hspace{-0.15mm}\ref{lem:approx_trans_formula}), \hspace{-0.15mm}the \hspace{-0.15mm}point-wise \hspace{-0.15mm}Poincar\'e~\hspace{-0.15mm}\mbox{inequality} (\textit{cf}.\ Lemma \ref{lem:poincare}), and that $k(s)\cdot n(s)+\tau R_\varepsilon(s,\tau)\ge \kappa- \varepsilon\|\mathtt{d}\|_{\infty,\Gamma_I}\|R_\varepsilon\|_{\infty,D_{I}^{\varepsilon}}$ for a.e. $\smash{(\tau,s)^\top}\in D_I^{\varepsilon}$ together with the approximative transformation formula (\textit{cf}.\ Lemma \ref{lem:approx_trans_formula}),~we~obtain
         \begin{align}
            \|v_\varepsilon\|_{\Sigma_{I}^{\varepsilon}}^2&= \int_{\Gamma_I}{
            \int_0^{\varepsilon \mathtt{d}(s)}{\vert v_\varepsilon (s+t k(s))\vert^2\big\{k(s)\cdot n(s)+tR_\varepsilon(s,t)\big\}\,\mathrm{d}t}\,\mathrm{d}s}\notag
            \\&\leq  \int_{\Gamma_I}
            \int_0^{\varepsilon \mathtt{d}(s)}\bigg\{(\varepsilon \mathtt{d}(s)-t)
            \int_t^{\varepsilon \mathtt{d}(s)}{\vert \nabla v_\varepsilon (s+\tau k(s))\vert^2\,\mathrm{d}\tau}\bigg\}\big\{k(s)\cdot n(s)+tR_\varepsilon(s,t)\big\}\,\mathrm{d}t\,\mathrm{d}s\notag
            \\&\leq  \varepsilon^2\|\mathtt{d}\|_{\infty,\Gamma_I}^2\big\{1+\varepsilon\|\mathtt{d}\|_{\infty,\Gamma_I}\|R_\varepsilon\|_{\infty,D_{I}^{\varepsilon}}\big\}\int_{\Gamma_I}{
            \int_0^{\varepsilon \mathtt{d}(s)}{\vert \nabla v_\varepsilon (s+\tau k(s))\vert^2\,\mathrm{d}\tau}\,\mathrm{d}s}\label{lem:equicoercive.2}
            \\&\leq  \varepsilon^2\|\mathtt{d}\|_{\infty,\Gamma_I}^2\big\{1+\varepsilon\|\mathtt{d}\|_{\infty,\Gamma_I}\|R_\varepsilon\|_{\infty,D_{I}^{\varepsilon}}\big\}\int_{\Gamma_I}{
            \int_0^{\varepsilon \mathtt{d}(s)}{\vert \nabla v_\varepsilon (s+\tau k(s))\vert^2\tfrac{k(s)\cdot n(s)+\tau R_\varepsilon(s,\tau)}{\kappa-\varepsilon\|\mathtt{d}\|_{\infty,\Gamma_I}\|R_\varepsilon\|_{\infty,D_{I}^{\varepsilon}}}\,\mathrm{d}\tau}\,\mathrm{d}s}\notag
            \\&\leq  \varepsilon^2\|\mathtt{d}\|_{\infty,\Gamma_I}^2
            \tfrac{1+\varepsilon\|\mathtt{d}\|_{\infty,\Gamma_I}\|R_\varepsilon\|_{\infty,D_{I}^{\varepsilon}}}{\kappa-\varepsilon\|\mathtt{d}\|_{\infty,\Gamma_I}\|R_\varepsilon\|_{\infty,D_{I}^{\varepsilon}}} 
            \|\nabla v_\varepsilon \|_{\Sigma_{I}^{\varepsilon}}^2\,,\notag
        \end{align}
        where $R_\varepsilon\in  L^\infty(D_{I}^{\varepsilon})$, $\varepsilon\in (0,\varepsilon_0)$, is as in Lemma \ref{lem:approx_trans_formula}.\newpage
        \noindent Using the point-wise Poincar\'e inequality (\textit{cf}.\ Lemma \ref{lem:poincare})
        and that $k(s)\cdot n(s)+t R_\varepsilon(s,t)\ge \kappa- \varepsilon\|\mathtt{d}\|_{\infty,\Gamma_I}\|R_\varepsilon\|_{\infty,D_{I}^{\varepsilon}}$ for a.e. $\smash{(t,s)^\top}\in D_I^{\varepsilon}$ together with the approximative transformation~formula (\textit{cf}.\ Lemma \ref{lem:approx_trans_formula}), we obtain
        \begin{align}\label{lem:equicoercive.3}
                 \|v_\varepsilon\|_{\Gamma_I}^2&\leq \int_{\Gamma_I}{\varepsilon \mathtt{d}(s)
            \int_0^{\varepsilon \mathtt{d}(s)}{\vert \nabla v_\varepsilon (s+t k(s))\vert^2\,\mathrm{d}t}\,\mathrm{d}s}
            \\&\leq \varepsilon\|\mathtt{d}\|_{\infty,\Gamma_I}\int_{\Gamma_I}{
            \int_0^{\varepsilon \mathtt{d}(s)}{\vert \nabla v_\varepsilon (s+t k(s))\vert^2\tfrac{k(s)\cdot n(s)+tR_\varepsilon(s,t)}{\kappa-\varepsilon \|\mathtt{d}\|_{\infty,\Gamma_I}\|R_\varepsilon\|_{\infty,D_{I}^{\varepsilon}}} \,\mathrm{d}t}\,\mathrm{d}s}\notag 
            \\&= \smash{\tfrac{\varepsilon\|\mathtt{d}\|_{\infty,\Gamma_I}}{\kappa-\varepsilon \|\mathtt{d}\|_{\infty,\Gamma_I}\|R_\varepsilon\|_{\infty,D_{I}^{\varepsilon}}}}\|\nabla v_\varepsilon\|_{\Sigma_{I}^{\varepsilon}}^2\,.\notag 
        \end{align}
        Then, using Friedrichs' inequality \eqref{lem:poin_cont}, Hölder's inequality, and \eqref{lem:equicoercive.3}, we infer that
        \begin{align}\label{lem:equicoercive.4}
            \begin{aligned}
            \|v_\varepsilon\|_{\Omega}^2&\leq c_{\textrm{F}}\,\big\{\|\nabla v_\varepsilon\|_{\Omega}^2+\tfrac{1}{\vert \Gamma_I\vert}\| v_\varepsilon\|_{\Gamma_I}^2\big\}
            \\&
            \leq c_{\textrm{F}}\,\big\{\|\nabla v_\varepsilon\|_{\Omega}^2+\tfrac{1}{\vert \Gamma_I\vert}\smash{\tfrac{\varepsilon\|\mathtt{d}\|_{\infty,\Gamma_I}}{\kappa-\varepsilon \|\mathtt{d}\|_{\infty,\Gamma_I}\|R_\varepsilon\|_{\infty,D_{I}^{\varepsilon}}}}\|\nabla v_\varepsilon\|_{\Sigma_{I}^{\varepsilon}}^2\big\}\,.
            \end{aligned}
        \end{align}
        Moreover, using the trace theorem (\textit{cf}.~\cite[Thm.~II.4.3]{Galdi}) and \eqref{lem:equicoercive.4}, we infer that
        \begin{align}\label{lem:equicoercive.5}
            \begin{aligned}
            \| v_\varepsilon\|_{\smash{H^{\smash{\frac{1}{2}}}(\Gamma_N)}}^2&\leq c_{\textrm{Tr}}\,\big\{\|\nabla v_\varepsilon\|_{\Omega}^2+
            \|v_\varepsilon\|_{\Omega}^2\big\}
            \\&\leq c_{\textrm{Tr}}\,\big\{(1+c_{\textrm{F}})\,\|\nabla v_\varepsilon\|_{\Omega}^2+c_{\textrm{F}}\tfrac{1}{\vert \Gamma_I\vert}\smash{\tfrac{\varepsilon\|\mathtt{d}\|_{\infty,\Gamma_I}}{\kappa-\varepsilon \|\mathtt{d}\|_{\infty,\Gamma_I}\|R_\varepsilon\|_{\infty,D_{I}^{\varepsilon}}}}\|\nabla v_\varepsilon\|_{\Sigma_{I}^{\varepsilon}}^2\big\}\,.
            \end{aligned}
        \end{align}
        In summary, using \eqref{lem:equicoercive.4} and \eqref{lem:equicoercive.5} in \eqref{lem:equicoercive.1}, for every $\delta>0$, we arrive at
        \begin{align}\label{lem:equicoercive.6}
            \tfrac{1}{2}\|\nabla v_\varepsilon\|_{\Omega}^2+\tfrac{\varepsilon}{2}\|\nabla v_\varepsilon\|_{\Sigma_{I}^{\varepsilon}}^2&\leq c_E+\tfrac{1}{2\delta}\big\{\|f\|_{\Omega}^2+\| g\|_{\smash{(H^{\smash{\frac{1}{2}}}(\Gamma_N))^*}}^2\big\}\\&\quad+\tfrac{\delta}{2}\max\big\{c_{\textrm{F}}+c_{\textrm{Tr}}\,(1+c_{\textrm{F}}),(1+c_{\textrm{Tr}})c_{\textrm{F}}\tfrac{1}{\vert \Gamma_I\vert}\smash{\tfrac{\|\mathtt{d}\|_{\infty,\Gamma_I}}{\kappa-\varepsilon \|\mathtt{d}\|_{\infty,\Gamma_I}\|R_\varepsilon\|_{\infty,D_{I}^{\varepsilon}}}}\big\}\notag
            \\&\quad\quad\times\big\{
           \|\nabla v_\varepsilon\|_{\Omega}^2
            +\varepsilon\|\nabla v_\varepsilon\|_{\Sigma_{I}^{\varepsilon}}^2\big\}\,.\notag
        \end{align}
        Then, choosing in \eqref{lem:equicoercive.6}
        \begin{align*}
            \tfrac{1}{\delta}=\tfrac{1}{\delta_\varepsilon}\coloneqq 2\max\big\{c_{\textrm{F}}+c_{\textrm{Tr}}\,(1+c_{\textrm{F}}),(1+c_{\textrm{Tr}})c_{\textrm{F}}\tfrac{1}{\vert \Gamma_I\vert}\tfrac{\|\mathtt{d}\|_{\infty,\Gamma_I}}{\kappa-\varepsilon \|\mathtt{d}\|_{\infty,\Gamma_I}\|R_\varepsilon\|_{\infty,D_{I}^{\varepsilon}}}\big\}
        >0\,,
        \end{align*}
        we conclude that
        \begin{align}\label{lem:equicoercive.7}
        \|\nabla v_\varepsilon\|_{\Omega}^2+\varepsilon\|\nabla v_\varepsilon\|_{\Sigma_{I}^{\varepsilon}}^2
        \leq 4c_E+\tfrac{2}{\delta_\varepsilon}\big\{\|f\|_{\Omega}^2+\| g\|_{\smash{(H^{\smash{\frac{1}{2}}}(\Gamma_N))^*}}^2\big\} \,.
        \end{align}
        Eventually, using \eqref{lem:equicoercive.7} together with $\sup_{\varepsilon\in (0,\varepsilon_0)}{\{\frac{1}{\delta_\varepsilon}\}}<+\infty$ in both \eqref{lem:equicoercive.2} and \eqref{lem:equicoercive.4},~we~conclude that 
        the claimed equi-coercivity property applies.
    \end{proof}

    \subsection{Transversal distance function}

    \hspace{5mm}In order to prove the $\limsup$-estimate in the later $\Gamma$-convergence result, we need to measure the distance of points in the thin {insulating} layer $\Sigma_{I}^{\varepsilon}$ to the {insulated} boundary $\Gamma_I$ with respect to the Lipschitz continuous (globally) transversal vector field $k\in (C^{0,1}(\partial\Omega))^d$.
    
    \begin{lemma}\label{lem:transversal_distance_function}
        For each $\varepsilon\in (0,\varepsilon_0)$, let the \emph{transversal distance function} $\psi_\varepsilon\colon \Sigma_{I}^{\varepsilon}\to [0,\varepsilon \|\mathtt{d}\|_{\infty,\Gamma_I})$, for every $x=s+tk(s)\in \Sigma_{I}^{\varepsilon}$, where $s\in \Gamma_I$ and $t\in [0,\varepsilon \mathtt{d}(s))$, be defined by\enlargethispage{5mm}
        \begin{align*}
            \psi_\varepsilon(x)\coloneqq t\,.
        \end{align*}
        Then, we have that $\psi_\varepsilon\in H^{1,\infty}(\Sigma_{I}^{\varepsilon})$ with  $\psi_\varepsilon=0$ a.e.\ on $\Gamma_I$ and 
        \begin{subequations} 
        \begin{align}\label{lem:transversal_distance_function.1}
        \|\psi_\varepsilon \|_{\infty,\Sigma_{I}^{\varepsilon}}&\leq \varepsilon\|\mathtt{d}\|_{\infty,\Gamma_I}\,,\\
            \nabla \psi_\varepsilon(x)&= \tfrac{1}{k(s)\cdot n(s)}n(s)+tR_\varepsilon(x)\quad \text{ for a.e.\ }x=s+tk(s)\in \Sigma_{I}^{\varepsilon}\,,\label{lem:transversal_distance_function.2}
        \end{align}
         \end{subequations}
        where $R_\varepsilon\in (L^\infty(\Sigma_{I}^{\varepsilon}))^d$, $\varepsilon\in (0,\varepsilon_0)$, depend only on the Lipschitz characteristics of $\partial\Omega$ and $k$ and satisfy $\sup_{\varepsilon\in (0,\varepsilon_0)}{\{\|R_\varepsilon\|_{\infty,\Sigma_{I}^{\varepsilon}}\}}<+\infty$.
    \end{lemma}

    \newpage
    \begin{proof}
     The estimate  \eqref{lem:transversal_distance_function.1} is evident. For the representation \eqref{lem:transversal_distance_function.2}, we first observe that
        \begin{align*}
            \psi_\varepsilon= \pi_{d+1}\circ\Phi_\varepsilon^{-1}\quad\text{ in }\Sigma_{I}^{\varepsilon}\,,
        \end{align*}
         where $\pi_{d+1}\colon \mathbb{R}^{d+1}\to \mathbb{R}$,  for every $x=(x_1,\ldots,x_{d+1})^\top  \in \mathbb{R}^{d+1}$, defined by $\pi_{d+1}(x)\coloneqq x_{d+1}$, denotes the projection onto the $(d+1)$-th component.
        Since $\Phi_\varepsilon\colon \hspace{-0.1em}D_{I}^{\varepsilon}\hspace{-0.1em}\to \hspace{-0.1em}\Sigma_{I}^{\varepsilon}$ is bi-Lipschitz~\mbox{continuous} and $\pi_{d+1}\colon \hspace{-0.15em}\mathbb{R}^{d+1}\hspace{-0.15em}\to \hspace{-0.15em}\mathbb{R}$ is Lipschitz continuous, 
        $\psi_\varepsilon\colon \hspace{-0.15em}\Sigma_{I}^{\varepsilon}\hspace{-0.15em}\to\hspace{-0.15em} [0,\varepsilon \|\mathtt{d}\|_{\infty,\Gamma_I})$ is Lipschitz~\mbox{continuous}~as~well. Hence, by Rademacher's theorem (\textit{cf}.\ \cite[Thm.\ 2.14]{AFP2000}), we have that $\psi_\varepsilon\in H^{1,\infty}(\Sigma_{I}^{\varepsilon})$.~\mbox{Moreover}, as in the proof of Lemma \ref{lem:approx_trans_formula},
         there exist some $r>0$ and a finite number $N\in \mathbb{N}$ of affine isometric mappings $A_i\colon \mathbb{R}^d\to \mathbb{R}^d$, $i=1,\ldots,N$, (\textit{i.e.}, $\mathrm{D}A_i\equiv O_i\in \mathrm{O}(d)$~for~all~$i=1,\ldots,N$) and Lipschitz  mappings
        $\gamma_i\colon  B_r\coloneqq B_r^{d-1}(0)\to \mathbb{R}$, $i=1,\ldots,N$,~such~that~$ \partial\Omega=\bigcup_{i=1}^N{A_i(\mathrm{graph}(\gamma_i))}$.~Let~$i\in \{1,\ldots,N\}$ be arbitrary and 
        let  $F_\varepsilon^{i}\colon U_\varepsilon^{i}\to \Sigma_{I}^{\varepsilon}$~be~defined~by~\eqref{lem:approx_trans_formula.2}. Then, for a.e.\ $(\overline{x},t)^\top\in U_\varepsilon^{i}$, we observe that\enlargethispage{5mm} 
        \begin{align} 
            \begin{aligned} 
            \left[\begin{array}{c|c}
                O_i & 0_d \\  \cmidrule(){1-2}
                0_d^\top \vphantom{X^{X^X}} & 1
            \end{array}\right]\left[\begin{array}{c|c}
                \mathrm{I}_{(d-1)\times (d-1)} & \multirow{2}{*}{$0_d$} \\\cmidrule(){1-1}
                \nabla \gamma_i(\overline{x})^\top\vphantom{X^{X^X}} & \\ \cmidrule(){1-2}
                \smash{0_{d-1}^\top}\vphantom{0_d^\top X^{X^X}} & 1
            \end{array}\right]&=\mathrm{D}(\Phi_\varepsilon^{-1}\circ F_\varepsilon^{i})(\overline{x},t)\\&=
            \mathrm{D}\Phi_\varepsilon^{-1}(F_\varepsilon^{i}(\overline{x},t))\mathrm{D}F_\varepsilon^{i}(\overline{x},t)\,. \end{aligned}\label{lem:transversal_distance_function.3}
        \end{align}
        In addition, for a.e.\ $(\overline{x},t)^\top\in U_\varepsilon^{i}$,  employing the abbreviations
        \begin{align*}
            s_i(\overline{x})&\coloneqq A_i(\overline{x},\gamma_{i}(\overline{x}))\,,\\
            d_i(\overline{x})&\coloneqq -(k(s_i(\overline{x}))\cdot n(s_i(\overline{x})))J_{\gamma_i}(\overline{x})\,,\\
            \overline{k}_i(\overline{x})&\coloneqq \pi_{\mathbb{R}^{d-1}}(O^\top_i k(s_i(\overline{x})))\,,
        \end{align*} 
        where $\pi_{\mathbb{R}^{d-1}}\colon \hspace{-0.1em}\mathbb{R}^d\hspace{-0.1em}\to\hspace{-0.1em} \mathbb{R}^{d-1}$, for $x\hspace{-0.1em}=\hspace{-0.1em}(x_1,\ldots,x_d)^\top \hspace{-0.1em} \in\hspace{-0.1em} \mathbb{R}^d$,~defined~by $\pi_{\mathbb{R}^{d-1}}(x)\coloneqq (x_1,\ldots,x_{d-1})^\top$, 
        denotes the projection~onto~$\mathbb{R}^{d-1}$, from \eqref{lem:approx_trans_formula.3}, using transversality and the block matrix inverse at $t=0$ (\textit{cf}.\ \cite[Thm.\ 2.1]{LuShiou2002}) and $A^{-1}-B^{-1}=A^{-1}(B-A)B^{-1}$, we deduce that
        \begin{align}\label{lem:transversal_distance_function.4}
        \begin{aligned} 
            (\mathrm{D}F_\varepsilon^{i}(\overline{x},t))^{-1}&=\frac{1}{d_i(\overline{x})}
            \left[\begin{array}{c|c}
                d_i(\overline{x})\mathrm{I}_{(d-1)\times (d-1)} +\overline{k}_i(\overline{x})\otimes \nabla \gamma_{i}(\overline{x}) & -\overline{k}_i(\overline{x}) \\\cmidrule(){1-2}
               - \nabla \gamma_{i}(\overline{x})^\top\vphantom{X^{X^X}} & 1
            \end{array}\right]O_i^{\top}\\&
            \quad+t\,R_\varepsilon^{i}(\overline{x},t)\,,
            \end{aligned}
        \end{align}
        where $R_\varepsilon^{i} \in (L^\infty(U_\varepsilon^{i}))^{d\times d}$, $\varepsilon\in (0,\varepsilon_0)$, depends on the Lipschitz geometry of $\partial\Omega$ and on $k$ and satisfies $\sup_{\varepsilon\in (0,\varepsilon_0)}{\{\|R_\varepsilon^{i}\|_{\infty,U_\varepsilon^{i}}\}}<+\infty$.
        For a.e.\ $(\overline{x},t)^\top\in U_\varepsilon^{i}$, using \eqref{lem:transversal_distance_function.4}~in~\eqref{lem:transversal_distance_function.3},~we~infer~that
        \begin{align}\label{lem:transversal_distance_function.5}
            \mathrm{D}\Phi_\varepsilon^{-1}(F_\varepsilon^{i}(\overline{x},t))&=\left[\begin{array}{c|c}
                O_i & 0_d \\  \cmidrule(){1-2}
                0_d^\top \vphantom{X^{X^X}} & 1
            \end{array}\right]\left[\begin{array}{c|c}
                \mathrm{I}_{(d-1)\times (d-1)} & \multirow{2}{*}{$0_d$} \\\cmidrule(){1-1}
                \nabla \gamma_{i}(\overline{x})^\top\vphantom{X^{X^X}} & \\ \cmidrule(){1-2}
                \smash{0_{d-1}^\top}\vphantom{0_d^\top X^{X^X}} & 1
            \end{array}\right]\\&\quad\times\frac{1}{d_i(\overline{x})}\left[\begin{array}{c|c}
                d_i(\overline{x})\mathrm{I}_{(d-1)\times (d-1)} +\overline{k}_i(\overline{x})\otimes \nabla \gamma_{i}(\overline{x}) & -\overline{k}_i(\overline{x}) \\\cmidrule(){1-2}
               - \nabla \gamma_{i}(\overline{x})^\top\vphantom{X^{X^X}} & 1
            \end{array}\right]O_i^{\top}+t\,\widetilde{R}_\varepsilon^{i}(\overline{x},t)\,,\notag
        \end{align}
        where $\smash{\widetilde{R}}_\varepsilon^{i}\in (L^\infty(U_\varepsilon^{i}))^{(d+1)\times d}$, $\varepsilon\in (0,\varepsilon_0)$, depends on the Lipschitz geometry of $\partial\Omega$ and on $k$~and satisfies $\sup_{\varepsilon\in (0,\varepsilon_0)}{\{\|\smash{\widetilde{R}}_\varepsilon^{i}\|_{\infty,U_\varepsilon^{i}}\}}<+\infty$.
        Due to $\mathrm{D}\pi_{d+1}=(0_d^\top,1)\in \mathbb{R}^{1\times (d+1)}$, for a.e.\ ${(\overline{x},t)^\top\in U_\varepsilon^{i}}$, from \eqref{lem:transversal_distance_function.5}, we deduce that
        \begin{align}\label{lem:transversal_distance_function.6}
            \begin{aligned} 
          \mathrm{D}\psi_{\varepsilon}(F_\varepsilon^{i}(\overline{x},t))&=\mathrm{D}\pi_{d+1}\mathrm{D}\Phi_\varepsilon^{-1}(F_\varepsilon^{i}(\overline{x},t))
          \\&=\tfrac{1}{d_i(\overline{x})}(-\nabla \gamma_{i}(\overline{x})^\top, 1) O_i^{\top}+t\,\mathrm{D}\pi_{d+1}\widetilde{R}_\varepsilon^{i}(\overline{x},t)
          \\&= \tfrac{1}{k(s_i(\overline{x}))\cdot n(s_i(\overline{x}))}\tfrac{1}{J_{\gamma_i}(\overline{x})}\big\{O_i(\nabla \gamma_{i}(\overline{x})^\top, -1)^\top\smash{\big\}^\top}+t\,\mathrm{D}\pi_{d+1}\widetilde{R}_\varepsilon^{i}(\overline{x},t)
          \\&=\tfrac{1}{k(s_i(\overline{x}))\cdot n(s_i(\overline{x}))}n(s_i(\overline{x}))^{\top}+t\,\mathrm{D}\pi_{d+1}\widetilde{R}_\varepsilon^{i}(\overline{x},t)\,.
          \end{aligned}
        \end{align}
        Eventually, since $i\in \{1,\ldots,N\}$ was chosen arbitrarily and $\Gamma_I\subseteq\smash{\bigcup_{i=1}^N{A_i(\mathrm{graph}(\gamma_i))}}$, from \eqref{lem:transversal_distance_function.6}, we conclude that  the claimed representation \eqref{lem:transversal_distance_function.2}~applies. 
    \end{proof}

    \newpage
    \section{$\Gamma$-convergence result}\label{sec:gamma_convergence}

    \hspace{5mm}Eventually, we have everything at our disposal to establish the main result of the paper,~\textit{i.e.},~the $\Gamma(L^2(\mathbb{R}^d))$-convergence (as $\varepsilon\hspace{-0.1em}\to\hspace{-0.1em} 0^+$) of the family of extended functionals $\smash{\overline{E}}_\varepsilon^\mathtt{d}\colon \hspace{-0.1em}\smash{L^2(\mathbb{R}^d)}\hspace{-0.1em}\to\hspace{-0.1em} \mathbb{R}\cup\{+\infty\}$, $\varepsilon\in (0,\varepsilon_0)$, defined as follows (with restrictions to the respective domains understood)\vspace{-0.5mm}
    \begin{align}
       \smash{\overline{E}}_\varepsilon^{\mathtt{d}}(v_\varepsilon)\coloneqq \begin{cases}
            \smash{E}_\varepsilon^{\mathtt{d}}(v_\varepsilon)&\text{ if }v_\varepsilon\in H^1(\Omega_{I}^{\varepsilon})\,,\\[-0.5mm]
            +\infty & \text{ else}\,,
        \end{cases}
    \end{align}
    to the extended functional 
    $\smash{\overline{E}}^\mathtt{d}\colon \smash{L^2(\mathbb{R}^d)}\to \mathbb{R}\cup\{+\infty\}$,  for every $v\in \smash{L^2(\mathbb{R}^d)}$ defined by\vspace{-0.5mm}
    \begin{align}
        \smash{\overline{E}}^{\mathtt{d}}(v)\coloneqq \begin{cases}
            \smash{E}^{\mathtt{d}}(v)& \text{ if }v\in H^1(\Omega)\,,\\[-0.5mm]
            +\infty& \text{ else}\,.
        \end{cases}
    \end{align}

    \begin{theorem}\label{thm:main}
       \hspace{-0.1mm}Let \hspace{-0.1mm}$\Gamma_I$ \hspace{-0.1mm}be \hspace{-0.1mm}piece-wise \hspace{-0.1mm}flat, \hspace{-0.1mm}\textit{i.e.}, there are disjoint, relatively open flat parts $\smash{\Gamma_I^\ell}$, $\ell=1,\ldots,L$, with constant outward unit normals $n_\ell$ and $\bigcup_{\ell=1}^L\Gamma_I^\ell=\Gamma_I$ up to a surface null set. Then, if $\mathtt{d}\in \smash{C^{0,1}(\Gamma_I)}$ with $\mathtt{d}\ge\smash{\mathtt{d}_{\textup{min}}}$ in $\Gamma_I$, for some $\smash{\mathtt{d}_{\textup{min}}}>0$, there holds\vspace{-0.5mm}
        \begin{align*}
            \Gamma(L^2(\mathbb{R}^d))\text{-}\lim_{\varepsilon\to 0^+}{\smash{\bigl\{\smash{\overline{E}}_\varepsilon^{\mathtt{d}}\bigr\}}}= \smash{\overline{E}}^{\mathtt{d}}\,,
        \end{align*}
        \textit{i.e.}, the following two statements apply:
        \begin{itemize}[noitemsep,topsep=2pt,leftmargin=!,labelwidth=\widthof{$\bullet$}] 
        \item[$\bullet$] \textit{$\liminf$-estimate.} For every sequence $\smash{(v_\varepsilon)_{\varepsilon\in (0,\varepsilon_0)}}\subseteq \smash{L^2(\mathbb{R}^d)}$ and $v\in \smash{L^2(\mathbb{R}^d)}$, from $v_\varepsilon\to v$ in $\smash{L^2(\mathbb{R}^d)}$ $(\varepsilon\to 0^{+})$,
        it follows that\vspace{-0.5mm}
        \begin{align*}
            \liminf_{\varepsilon\to 0^+}{\smash{\bigl\{\smash{\overline{E}}_\varepsilon^{\mathtt{d}}(v_\varepsilon)\bigr\}}}\ge \smash{\overline{E}}^{\mathtt{d}}(v)\,;
        \end{align*}

        \item[$\bullet$] \textit{$\limsup$-estimate.} For every $v\in \smash{L^2(\mathbb{R}^d)}$, there exists a sequence $\smash{(v_\varepsilon)_{\varepsilon\in (0,\varepsilon_0)}}\subseteq \smash{L^2(\mathbb{R}^d)}$~such~that  $v_\varepsilon\to v$ in $\smash{L^2(\mathbb{R}^d)}$ $(\varepsilon\to 0^{+})$ and\vspace{-0.5mm}
        \begin{align*}
         \limsup_{\smash{\varepsilon\to 0^+}}{\smash{\bigl\{\smash{\overline{E}}_\varepsilon^{\mathtt{d}}(v_\varepsilon)\bigr\}}}\leq \smash{\overline{E}}^{\mathtt{d}}(v)\,.
        \end{align*}
    \end{itemize}
    \end{theorem}

    \begin{remark}
        {We emphasize that the assumption $\mathtt{d}\ge \mathtt{d}_{\textup{min}}$ in $\Gamma_I$, for some $\mathtt{d}_{\textup{min}}>0$, in some cases, when the total amount of insulating material $m>0$ is small, may not be satisfied~(\textit{cf}.~\cite{Buttazzo1988c}).}
    \end{remark}
    
    Let us start by proving the $\liminf$-estimate.\vspace{-1mm}\enlargethispage{5mm}

    \begin{lemma}[$\liminf$-estimate]\label{lem:liminf}
        Let $\Gamma_I$ be piece-wise flat.
        Then, if $\mathtt{d}\in C^{0,1}(\Gamma_I)$ with $\mathtt{d}\ge \mathtt{d}_{\textup{min}}$ in $\Gamma_I$, for some $\mathtt{d}_{\textup{min}}>0$, for every sequence $\smash{(v_\varepsilon)_{\varepsilon\in (0,\varepsilon_0)}}\subseteq L^2(\mathbb{R}^d)$ and $v\in L^2(\mathbb{R}^d)$, from $v_\varepsilon\to v$ in $L^2(\mathbb{R}^d)$ $(\varepsilon\to 0)$, 
        it follows that
        \begin{align*}
            \liminf_{\varepsilon\to 0^+}{\bigl\{\smash{\overline{E}}_\varepsilon^{\mathtt{d}}(v_\varepsilon)\bigr\}}\ge \smash{\overline{E}}^{\mathtt{d}}(v)\,.
        \end{align*}
    \end{lemma}

    \begin{proof}
         Let $\smash{(v_\varepsilon)_{\varepsilon\in (0,\varepsilon_0)}}\subseteq L^2(\mathbb{R}^d)$ be a sequence such that $v_\varepsilon\to v$ in $L^2(\mathbb{R}^d)$ $(\varepsilon\to 0^{+})$.
        Then, without loss of generality, we may assume that $\smash{\liminf_{\varepsilon\to 0^+}{\smash{\overline{E}}_\varepsilon^{\mathtt{d}}(v_\varepsilon)}<+\infty}$. Otherwise,~we~trivially have that $\liminf_{\varepsilon\to 0^+}{\smash{\overline{E}}_\varepsilon^{\mathtt{d}}(v_\varepsilon)}\ge \smash{\overline{E}}^{\mathtt{d}}(v)$.
        Hence, there exists a subsequence {$(v_{\varepsilon'})_{\varepsilon'\in (0,\varepsilon_0)}\subseteq L^2(\mathbb{R}^d)$ with $v_{\varepsilon'}\in H^1(\smash{\Omega_I^{\varepsilon'}}\vphantom{\Omega_{\varepsilon'}})$},
        $v_{\varepsilon'}=0$ a.e.\ on $\smash{\Gamma_I^{\varepsilon'}}$, and  $v_{\varepsilon'}=u_D$ a.e.\ on $\Gamma_D$ for all $\varepsilon'\in (0,\varepsilon_0)$ such that 
        \begin{align}\label{lem:liminf.1}
            E_{\varepsilon'}^{\mathtt{d}}(v_{\varepsilon'})\to \liminf_{\varepsilon\to 0^+}{\bigl\{\smash{\overline{E}}_\varepsilon^{\mathtt{d}}(v_\varepsilon)\bigr\}}\quad (\varepsilon'\to 0^{+})\,.
        \end{align}
        From \eqref{lem:liminf.1}, by Lemma \ref{lem:equicoercive} applied to the restrictions of $v_{\varepsilon'}$ to $\smash{\Omega_I^{\varepsilon'}}\vphantom{L^2(\mathbb{R}^d)}$, we obtain\vspace{-0.5mm}
        \begin{align*}
            \sup_{\varepsilon'\in (0,\varepsilon_0)}\smash{\big\{\|v_{\varepsilon'}\|_{\Omega}^2+\|\nabla v_{\varepsilon'}\|_{\Omega}^2+\tfrac{1}{\varepsilon'}\|v_{\varepsilon'}\|_{{\Sigma_{I}^{\smash{\varepsilon'}}}}^2+\varepsilon'\|\nabla v_{\varepsilon'}\|_{{\Sigma_{I}^{\smash{\varepsilon'}}}}^2\big\}}<+\infty\,,\\[-6mm]
        \end{align*}
        which, using the weak continuity of the trace operator from $\smash{H^1(\Omega)}$ to $\smash{H^{\smash{\frac{1}{2}}}(\partial\Omega)}$ (\textit{cf}.~\cite[Thm.~II.4.3]{Galdi}) and the compact embedding $H^{\smash{\frac{1}{2}}}(\partial\Omega)\hookrightarrow L^2(\partial\Omega)$, implies that $v\in H^1(\Omega)$ and  
        \begin{alignat}{3}\label{lem:liminf.2}
          v_{\varepsilon'}&\rightharpoonup  v&&\quad \text{ in }H^1(\Omega)&&\quad(\varepsilon'\to0^{+})\,,\\[-0.25mm]\label{lem:liminf.2.1}v_{\varepsilon'}&\rightharpoonup  v&&\quad \text{ in }H^{\smash{\frac{1}{2}}}(\partial\Omega)&&\quad(\varepsilon'\to0^{+})\,,\\[-0.25mm]\label{lem:liminf.3}
            v_{\varepsilon'}&\to v&&\quad \text{ in }L^2(\partial\Omega)&&\quad(\varepsilon'\to0^{+})\,.
        \end{alignat}\newpage
    \begingroup\setdisplayskips{4.5pt}
        \noindent In particular, we have that $v=u_D$ a.e.\ on $\Gamma_D$. From \eqref{lem:liminf.2} and \eqref{lem:liminf.2.1}, in turn, we obtain 
        \begin{align}\label{lem:liminf.4}
            \smash{\liminf_{\varepsilon'\to 0^+}{\big\{\tfrac{1}{2}\|\nabla v_{\varepsilon'}\|_{\Omega}^2-(f,v_{\varepsilon'})_{\Omega}-\langle g,v_{\varepsilon'}\rangle_{\smash{H^{\smash{\frac{1}{2}}}(\Gamma_N)}}\big\}}}\ge \tfrac{1}{2}\|\nabla v\|_{\Omega}^2-(f,v)_{\Omega}-\langle g,v\rangle_{\smash{H^{\smash{\frac{1}{2}}}(\Gamma_N)}}\,.
        \end{align}
        Let $\Gamma_I^\ell$, $\ell=1,\ldots,L$, be the flat parts of Theorem \ref{thm:main}, with outward normals $n_\ell$. For each $\ell$, let $K_\ell$ be the part of $\smash{\Sigma_I^{\varepsilon_0}}\vphantom{\bigcup_{\ell=1}^L\Gamma_I^\ell}$ parametrized over $\Gamma_I^\ell$ in the direction of $k$.
        Then, for every $\ell=1,\ldots,L$, we introduce the transformation $\phi_{\varepsilon'}^{\ell}\colon \Gamma_{I}^{\ell}\to \mathbb{R}^d$ (\textit{cf}.\ Figure \ref{fig:construction}), for every $s\in \Gamma_{I}^{\ell}$ defined by\enlargethispage{5mm}
        \begin{align}\label{lem:liminf.5}
        \begin{aligned} 
            \phi_{\varepsilon'}^{\ell}(s)&\coloneqq s+\varepsilon' \mathtt{d}(s)k(s)-\varepsilon'  \widetilde{\mathtt{d}}(s)n_{\ell}
            \\&\,= s+\varepsilon' \mathtt{d}(s)\smash{\big\{k(s)-(k(s)\cdot n_{\ell})n_{\ell}\big\}} \,,
            \end{aligned}
        \end{align}
        which, for $\varepsilon'\in(0,\widetilde{\varepsilon}_0)$ with $\widetilde{\varepsilon}_0\in(0,\varepsilon_0)$ small, is bi-Lipschitz continuous.~Moreover,~in~\eqref{lem:liminf.5}, the distribution~function~$\widetilde{\mathtt{d}}\colon \Gamma_I\to [0,+\infty)$ (in direction of~$n$) 
        is defined by \eqref{eq:relation_h_h_tilde}.
        By definition \eqref{lem:liminf.5}, for every $\varepsilon'\in (0,\widetilde{\varepsilon}_0)$ and $\ell=1,\ldots,L$, we have that
        \begin{align}\label{lem:liminf.7} \|\mathrm{id}_{\mathbb{R}^d}-\phi_{\varepsilon'}^{\ell}\|_{\infty,\Gamma_{I}^{\ell}}\leq 2\|\mathtt{d}\|_{\infty,\Gamma_{I}^{\ell}}\,\varepsilon'\,,
        \end{align}
        For small $\varepsilon'\vphantom{\Gamma_I^{\varepsilon',\ell}}$, \eqref{lem:liminf.7} makes $s\mapsto\widetilde{s}+s-\phi_{\varepsilon'}^\ell(s)$ a self-map of a closed ball in $\Gamma_I^\ell$ about $\widetilde{s}$. Set
        \begin{align*}
            \smash{\Gamma_I^{\varepsilon',\ell}\coloneqq \big\{\widetilde{s}\in\Gamma_I^\ell\cap\phi_{\varepsilon'}^\ell(\Gamma_I^\ell)\mid \widetilde{s}+(0,\varepsilon'\|\mathtt{d}\|_{\infty,\Gamma_I}]n_\ell\subset K_\ell^\circ\big\}\,.}
        \end{align*} 
        Brouwer's theorem and the one-sided neighborhood of $\widetilde{s}$ in $K_\ell$ (by transversality) yield
        \begin{align}\label{lem:liminf.8}
            \smash{\chi_{\smash{\Gamma_{I}^{\varepsilon',\ell}}}\to  1\quad\text{ a.e.\ in }\Gamma_{I}^{\ell}\quad (\varepsilon'\to 0^{+})\,.}
        \end{align} 
        From \eqref{lem:liminf.7}, in turn, for every $\varepsilon'\in (0,\widetilde{\varepsilon}_0)$ and $\ell=1,\ldots,L$, we infer that
        \begin{align*}
            \|\mathrm{id}_{\mathbb{R}^d}-(\phi_{\varepsilon'}^{\ell})^{-1}\|_{\infty,\Gamma_{I}^{\varepsilon',\ell}}\leq\|\phi_{\varepsilon'}^{\ell}-\mathrm{id}_{\mathbb{R}^d}\|_{\infty,\Gamma_{I}^{\ell}}\leq 2\|\mathtt{d}\|_{\infty,\Gamma_{I}^{\ell}}\varepsilon'\,,
        \end{align*}
        which, on the basis of $\widetilde{\mathtt{d}}\in C^{0,1}(\Gamma_I^{\ell})$ (because $\mathtt{d}\in C^{0,1}(\Gamma_I^{\ell})$, $k\in (C^{0,1}(\Gamma_I^{\ell}))^d$, and $n=n_{\ell}$~in~$\Gamma_I^{\ell}$) for all $\ell=1,\ldots,L$ and the representation \eqref{eq:relation_h_h_tilde}, implies that
        \begin{align}\label{lem:liminf.9}
            \|\widetilde{\mathtt{d}}\circ (\phi_{\varepsilon'}^{\ell})^{-1}- (k\cdot n)\mathtt{d}\|_{\infty,\Gamma_{I}^{\varepsilon',\ell}}\leq 2\operatorname{Lip}_{\Gamma_I^{\ell}}(\widetilde{\mathtt{d}})\|\mathtt{d}\|_{\infty,\Gamma_{I}^{\ell}}\varepsilon'\,.
        \end{align}
        Next, we define the \textit{thin {insulating}  layer in  direction of $n$} (\textit{cf}.\ Figure \ref{fig:construction})\vspace{-1.5mm}
        \begin{align}\label{def:sigma_tilde}
        \begin{aligned}
                \widetilde{\Sigma}_{I}^{\varepsilon'}&\coloneqq \bigcup_{\ell=1}^L{\widetilde{\Sigma}_{I}^{\varepsilon',\ell}}\,,\text{ where}\\
                \widetilde{\Sigma}_{I}^{\varepsilon',\ell}&\coloneqq \big\{ \widetilde{s}+t n_{\ell}\mid \widetilde{s}\in \Gamma_{I}^{\varepsilon',\ell}\,,\;t\in [0,\varepsilon'\widetilde{\mathtt{d}}((\phi_{\varepsilon'}^{\ell})^{-1}(\widetilde{s})))\big\}\quad \text{ for all }\ell=1,\ldots,L\,.
            \end{aligned}
        \end{align}  
        If we define the \textit{interacting insulation boundary parts in direction of $n$} (\textit{cf}.\ Figure \ref{fig:construction})
        \begin{align}\label{def:gamma_tilde}
            \widetilde{\Gamma}_{I}^{\varepsilon',\ell}&\coloneqq \big\{ \widetilde{s}+\varepsilon'\widetilde{\mathtt{d}}((\phi_{\varepsilon'}^{\ell})^{-1}(\widetilde{s}))\,n_{\ell}\mid \widetilde{s}\in \Gamma_{I}^{\varepsilon',\ell}\big\}\quad\text{ for all }\ell=1,\ldots,L\,,
        \end{align}
        then, for every $\ell=1,\ldots,L$, by definition \eqref{lem:liminf.5} (\textit{cf}.\ Figure \ref{fig:construction}), we have that
        \begin{align}\label{lem:liminf.10}
            \smash{\widetilde{\Gamma}_{I}^{\varepsilon',\ell}\subseteq \Gamma_I^{\varepsilon'}\,.}
        \end{align}
        The zero trace on $\Gamma_I^{\varepsilon'}$ permits extension to $H^1(\Omega_I^{\varepsilon_0})$ by gluing in the $k$-coordinates via
        \begin{align*}
            \smash{v_{\varepsilon'}\coloneqq 0\quad \text{ a.e.\ in }\Sigma_{I}^{\varepsilon_0}\setminus \Sigma_{I}^{\varepsilon'}\,.}
        \end{align*}
        By construction, $\smash[t]{\widetilde{\Sigma}_I^{\varepsilon',\ell}}\vphantom{\widetilde{\Sigma}_I^{\varepsilon'}}\subseteq K_\ell$, and the $K_\ell$ are disjoint by injectivity of the $k$-parametrization. Thus, for every $\varepsilon'\in(0,\widetilde{\varepsilon}_0)$, we have that
        \begin{align}\label{lem:liminf.11}
            \smash{\nabla v_{\varepsilon'}=0\quad \textrm{ a.e.\ in }\widetilde{\Sigma}_{I}^{\varepsilon'}\setminus\Sigma_{I}^{\varepsilon'}\,.}
        \end{align}
        Resorting to the point-wise Poincar\'e inequality (\textit{cf}.\ Lemma \ref{lem:poincare} with $\Sigma_{I}^{\varepsilon}\coloneqq \smash{\widetilde{\Sigma}_{I}^{\varepsilon',\ell}}$, \textit{i.e.},
        $\Gamma_I=\smash{\Gamma_{I}^{\varepsilon',\ell}}$, $\Gamma_I^{\varepsilon'}=\widetilde{\Gamma}_{I}^{\varepsilon',\ell}$,  $k=n_{\ell}$, $\mathtt{d}=\widetilde{\mathtt{d}}\circ (\phi_{\varepsilon'}^{\ell})^{-1}$, and $\varepsilon=\varepsilon'$), for a.e.\ $\widetilde{s}\in \Gamma_{I}^{\varepsilon',\ell}$, $\ell=1,\ldots,L$, $\varepsilon'\in  (0,\widetilde{\varepsilon}_0)$, due to $v_{\varepsilon'}(\widetilde{s}+\varepsilon'\widetilde{\mathtt{d}}((\phi_{\varepsilon'}^{\ell})^{-1}(\widetilde{s}))\,n_{\ell})=0$ (as \eqref{lem:liminf.10}), we find that
        \begin{align}
            \vert v_{\varepsilon'}(\widetilde{s})\vert^2\leq \varepsilon'\widetilde{\mathtt{d}}((\phi_{\varepsilon'}^{\ell})^{-1}(\widetilde{s}))\int_0^{\varepsilon'\widetilde{\mathtt{d}}((\phi_{\varepsilon'}^{\ell})^{-1}(\widetilde{s}))}{\vert \nabla v_{\varepsilon'}(\widetilde{s}+t n_\ell)\vert^2\,\mathrm{d}t}\,.
        \end{align}
        \setdisplayskips{4pt}
        \par\pagebreak\noindent
        Then, by the approximative transformation formula 
        (\textit{cf}.\ Lemma \ref{lem:approx_trans_formula}  with $\Sigma_{I}^{\varepsilon}\hspace{-0.1em}\coloneqq \hspace{-0.1em}\widetilde{\Sigma}_{I}^{\varepsilon',\ell}$,~\textit{i.e.},~${\Gamma_I\hspace{-0.1em}=\hspace{-0.1em}\Gamma_{I}^{\varepsilon',\ell}}$, $\Gamma_I^{\varepsilon'}=\widetilde{\Gamma}_{I}^{\varepsilon',\ell}$,  $k=n_{\ell}$, $\mathtt{d}=\widetilde{\mathtt{d}}\circ {(\phi_{\varepsilon'}^{\ell})^{-1}}$, and $\varepsilon=\varepsilon'$), for every $\ell=1,\ldots ,L$, we have that
        \begin{align}\label{lem:liminf.12}
            \begin{aligned} 
           \|\nabla v_{\varepsilon'}\|_{\widetilde{\Sigma}_{I}^{\varepsilon',\ell}}^2&=\int_{\Gamma_{I}^{\varepsilon',\ell}}{\int_0^{\varepsilon' \widetilde{\mathtt{d}}((\phi_{\varepsilon'}^{\ell})^{-1}(\widetilde{s}))}{ \vert \nabla v_{\varepsilon'}(\widetilde{s}+t n_{\ell})\vert^2\big\{1+t\widetilde{R}_{\varepsilon'}^{\ell}(\widetilde{s},t) \big\}\,\mathrm{d}t}\,\mathrm{d}\widetilde{s}}
           \\&\ge \tfrac{1}{\varepsilon'}\|(\widetilde{\mathtt{d}}\circ(\phi_{\varepsilon'}^{\ell})^{-1})^{-\smash{\frac{1}{2}}}v_{\varepsilon'}\|_{\Gamma_{I}^{\varepsilon',\ell}}^2\big\{1-\varepsilon'\|\widetilde{\mathtt{d}}\|_{\infty,\Gamma_I}\|\widetilde{R}_{\varepsilon'}^{\ell}\|_{\infty,{\widetilde{D}_{I}^{\varepsilon',\ell}}}\big\}\,,
            \end{aligned}
        \end{align}
        where $\widetilde{R}_{\varepsilon'}^{\ell}\in L^\infty(\widetilde{D}_{I}^{\varepsilon',\ell})$, $\widetilde{D}_{I}^{\varepsilon',\ell}\coloneqq \bigcup_{\widetilde{s}\in \smash{\Gamma_{I}^{\varepsilon',\ell}}}{\{\widetilde{s}\}\times [0,\varepsilon' \widetilde{\mathtt{d}}((\phi_{\varepsilon'}^{\ell})^{-1}(\widetilde{s})))}$, $\varepsilon'\in (0,\widetilde{\varepsilon}_0)$,  are~as~in~Lemma~\ref{lem:approx_trans_formula} with $\sup_{\varepsilon'\in (0,\widetilde{\varepsilon}_0)}{\{\|\widetilde{R}_{\varepsilon'}^{\ell}\|_{\infty,\widetilde{D}_{I}^{\varepsilon',\ell}}\}}<+\infty$. Then, using
        \eqref{lem:liminf.3}, \eqref{lem:liminf.8}, and \eqref{lem:liminf.9}, from \eqref{lem:liminf.12},  for every $\ell=1,\ldots ,L$, we deduce that
        \begin{align}\label{lem:liminf.13}
        \begin{aligned} 
            \liminf_{\varepsilon'\to 0^+}{\big\{\tfrac{\varepsilon'}{2}\|\nabla v_{\varepsilon'}\|_{\widetilde{\Sigma}_{I}^{\varepsilon',\ell}}^2\big\}}&\ge   \liminf_{\varepsilon'\to 0^+}{\big\{\tfrac{1}{2}\|(\widetilde{\mathtt{d}}\circ(\phi_{\varepsilon'}^{\ell})^{-1})^{-\smash{\frac{1}{2}}}v_{\varepsilon'}\chi_{\Gamma_{I}^{\varepsilon',\ell}}\|_{\Gamma_{I}^{\ell}}^2\big\}}
            \\&= \tfrac{1}{2} \|((k\cdot n)\mathtt{d})^{-\smash{\frac{1}{2}}}v\|_{\Gamma_I^{\ell}}^2
            \,.
            \end{aligned}
        \end{align}
        Eventually, from \eqref{lem:liminf.11}, \eqref{lem:liminf.13} and the facet decomposition of $\Gamma_I\vphantom{\bigcup_{\ell=1}^{L}}$, we obtain
        \begin{align}\label{lem:liminf.14}
        \begin{aligned} 
            \liminf_{\varepsilon'\to 0^+}{\big\{\tfrac{\varepsilon'}{2}\|\nabla v_{\varepsilon'}\|_{\Sigma_{I}^{\smash{\varepsilon'}}}^2\big\}}&\ge 
            \liminf_{\varepsilon'\to 0^+}{\big\{\tfrac{\varepsilon'}{2}\|\nabla v_{\varepsilon'}\|_{\widetilde{\Sigma}_{I}^{\smash{\varepsilon'}}}^2\big\}}
            \\[-0.5mm]&\ge \sum_{\ell=1}^L{\liminf_{\varepsilon'\to 0^+}{\big\{\tfrac{\varepsilon'}{2}\|\nabla v_{\varepsilon'}\|_{\widetilde{\Sigma}_{I}^{\varepsilon',\ell}}^2\big\}}}
            \\[-0.5mm]& \geq\sum_{\ell=1}^L{\tfrac{1}{2} \|((k\cdot n)\mathtt{d})^{-\smash{\frac{1}{2}}}v\|_{\Gamma_I^{\ell}}^2}
            \\&=\tfrac{1}{2} \|((k\cdot n)\mathtt{d})^{-\smash{\frac{1}{2}}}v\|_{\Gamma_I}^2\,.
            \end{aligned}
        \end{align}
        In summary, from \eqref{lem:liminf.4} and \eqref{lem:liminf.14}, we conclude that
        the claimed $\liminf$-estimate~applies.\enlargethispage{5mm}
    \end{proof}\vspace{-2.5mm}
    \par\endgroup

    \begingroup\setlength{\intextsep}{8pt}
    \begin{figure}[H]
        \captionsetup{skip=8pt}
        
\centering
  
\tikzset {_efbeu1xy7/.code = {\pgfsetadditionalshadetransform{ \pgftransformshift{\pgfpoint{0 bp } { 0 bp }  }  \pgftransformrotate{-90 }  \pgftransformscale{2 }  }}}
\pgfdeclarehorizontalshading{_1baele0ej}{150bp}{rgb(0bp)=(0.89,0.89,0.89);
rgb(37.5bp)=(0.89,0.89,0.89);
rgb(37.5bp)=(0.86,0.86,0.86);
rgb(37.5bp)=(0.82,0.82,0.82);
rgb(62.5bp)=(1,1,1);
rgb(100bp)=(1,1,1)}
\tikzset{every picture/.style={line width=0.75pt}} 


\caption{Sketch \hspace{-0.1mm}of \hspace{-0.1mm}the \hspace{-0.1mm}construction \hspace{-0.1mm}in \hspace{-0.1mm}the \hspace{-0.1mm}proof \hspace{-0.1mm}of \hspace{-0.1mm}Lemma \hspace{-0.1mm}\ref{lem:liminf}:
\hspace{-0.1mm}\textit{(a)}  \hspace{-0.1mm}parametrizations~\hspace{-0.1mm}${\phi_{\varepsilon'}^{\ell}\colon \hspace{-0.175em}\Gamma_I^{\ell}\hspace{-0.175em}\to\hspace{-0.175em} \mathbb{R}^d}$, $\ell=1,\ldots,L$, (\textit{cf}.\ \eqref{lem:liminf.5}); \textit{(b)}
 boundary parts $\Gamma_I^{\varepsilon',\ell}\subseteq \Gamma_I^{\ell}\cap \phi_{\varepsilon'}^{\ell}(\Gamma_I^{\ell}) $, $\ell=1,\ldots,L$; \textit{(c)}  thin {insulating} layer $\widetilde{\Sigma}_{I}^{\varepsilon'}$  (\textit{cf}.\ \eqref{def:sigma_tilde});  \textit{(d)}  interacting boundary parts $\widetilde{\Gamma}_I^{\varepsilon',\ell}$, $\ell=1,\ldots,L$, (\textit{cf}.\ \eqref{def:gamma_tilde}).}
 \label{fig:construction}
    \end{figure}\endgroup
    \newpage

    \begingroup\setdisplayskips{5pt}
    Let us continue by proving the $\limsup$-estimate, which does not require piece-wise flatness of the {insulated} boundary $\Gamma_I$.

    \begin{lemma}[$\limsup$-estimate]\label{lem:limsup}
        If $\mathtt{d}\hspace{-0.1em}\in\hspace{-0.1em} C^{0,1}(\Gamma_I)$ with $\mathtt{d}\hspace{-0.1em}\ge\hspace{-0.1em} \mathtt{d}_{\textup{min}}$, for some $\mathtt{d}_{\textup{min}}\hspace{-0.1em}>\hspace{-0.1em}0$,~then~for~every $v\hspace{-0.1em}\in\hspace{-0.1em} L^2(\mathbb{R}^d)$, there exists a sequence $\smash{(v_\varepsilon)_{\varepsilon\in (0,\varepsilon_0)}}\hspace{-0.1em}\subseteq\hspace{-0.1em} L^2(\mathbb{R}^d)$ such that $v_\varepsilon\hspace{-0.1em}\to\hspace{-0.1em} v$ in $L^2(\mathbb{R}^d)$ $(\varepsilon\hspace{-0.1em}\to\hspace{-0.1em} 0^{+})$~and
        \begin{align*}
            \limsup_{\smash{\varepsilon\to 0^+}}{\bigl\{\smash{\overline{E}}_\varepsilon^{\mathtt{d}}(v_\varepsilon)\bigr\}}\leq \smash{\overline{E}}^{\mathtt{d}}(v)\,.
        \end{align*}
    \end{lemma}
    
    \begin{proof} 
        Let $v\in L^2(\mathbb{R}^d)$ be arbitrary. Without loss of generality, we may assume that  $v|_\Omega\in H^1(\Omega)$ {with $v=u_D$ a.e.\ on  $\Gamma_D$}. Otherwise, we choose $v_\varepsilon\hspace{-0.1em}=\hspace{-0.1em}v\hspace{-0.1em}\in\hspace{-0.1em} L^2(\mathbb{R}^d)$ for all $\varepsilon\hspace{-0.1em}\in\hspace{-0.1em} (0,\varepsilon_0)$, which~implies~that $\smash{\overline{E}}_\varepsilon^{\mathtt{d}}(v_\varepsilon)=+\infty=\smash{\overline{E}}^{\mathtt{d}}(v)$ for all $\varepsilon\in (0,\varepsilon_0)$. Then, if, for every $x\in \overline{\Omega}_{I}^{\varepsilon}$, we define\enlargethispage{5mm} 
        \begin{align}\label{lem:limsup.1}
            \varphi_\varepsilon(x)\coloneqq 
            \begin{cases}
                1-\frac{\psi_\varepsilon(x)}{\varepsilon \mathtt{d}(x)}&\text{ if }x\in \smash{\overline{\Sigma_{I}^{\varepsilon}}}\,,\\
                1&\text{ if }x\in \overline{\Omega}\,,
            \end{cases}
        \end{align}
        where $\psi_\varepsilon$ is as in Lemma \ref{lem:transversal_distance_function} and $\mathtt{d}(s+tk(s))\coloneqq\mathtt{d}(s)$, 
        we have that $\varphi_\varepsilon\in   \smash{H^{1,\infty}(\Omega_{I}^{\varepsilon})}$ with
        \begin{subequations}\label{lem:limsup.2}
        \begin{alignat}{2}
            0\leq \varphi_\varepsilon&\leq 1 &&\quad\text{ in }\overline{\Omega}_{I}^{\varepsilon}\,,\label{lem:limsup.2.1}\\[-0.25mm]
            \varphi_\varepsilon&=1&&\quad\text{ in }\overline{\Omega}\,,\label{lem:limsup.2.2}\\[-0.25mm]
            \varphi_\varepsilon&=0 &&\quad\text{ on }\Gamma_{I}^{\varepsilon}\,.\label{lem:limsup.2.3}
           \\[-0.25mm]   \nabla \varphi_\varepsilon &= \smash{\tfrac{\psi_\varepsilon}{\varepsilon \mathtt{d}^2}\nabla \mathtt{d} -\tfrac{1}{\varepsilon \mathtt{d}}\nabla\psi_\varepsilon}&&\quad\text{ a.e.\ in }\Sigma_{I}^{\varepsilon}\,,\label{lem:limsup.2.4}
        \end{alignat}
        \end{subequations}
        \hspace{-0.1mm}By \hspace{-0.1mm}the \hspace{-0.1mm}Lipschitz \hspace{-0.1mm}regularity~\hspace{-0.1mm}of~\hspace{-0.1mm}the~\hspace{-0.1mm}domain~\hspace{-0.1mm}$\Omega$,  
        using the Sobolev extension theorem (\textit{cf}.~\cite[Thm.~II.3.3]{Galdi}), 
        we extend $v|_\Omega$ to $\overline{v}\in H^1(\mathbb{R}^d)$ for use in the insulating layer. Then, for every $\varepsilon\hspace{-0.1em}\in \hspace{-0.1em}(0,\varepsilon_0)$, we consider the function~$ v_\varepsilon\hspace{-0.1em}\in\hspace{-0.1em} L^2(\mathbb{R}^d)$, defined by $v_\varepsilon\coloneqq \overline{v}\varphi_\varepsilon$ a.e.\ in $\Omega_{I}^{\varepsilon}$ and $v_\varepsilon\coloneqq v$ a.e.\ in $\mathbb{R}^d\setminus\Omega_{I}^{\varepsilon} $, which satisfies 
       $v_\varepsilon|_{\Omega_{I}^{\varepsilon}} \in H^1(\Omega_{I}^{\varepsilon})$, and, by \eqref{lem:limsup.2.2} and \eqref{lem:limsup.2.3}, respectively,~that
        \begin{subequations}\label{lem:limsup.3}
        \begin{alignat}{2}\label{lem:limsup.3.1}
            v_\varepsilon&=v&&\quad\text{ a.e.\ in }\mathbb{R}^d\setminus \Sigma_{I}^{\varepsilon} \,,\\[-0.25mm]
            v_\varepsilon&=u_D&&\quad\text{ a.e.\ on }\smash{\Gamma_D} \,,\label{lem:limsup.3.2.0}\\[-0.25mm]
            v_\varepsilon&=0&&\quad\text{ a.e.\ on }\smash{\Gamma_I^\varepsilon} \,.\label{lem:limsup.3.2}
        \end{alignat}
        \end{subequations}
        Since $\smash{\vert v_\varepsilon-v\vert\leq(\vert\overline{v}\vert+\vert v\vert)\chi_{\Sigma_I^\varepsilon}}$ and $\smash{\vert\Sigma_I^\varepsilon\vert\to 0}$, we have that
        \begin{align}
            v_\varepsilon\to v\quad \text{ in }L^2(\mathbb{R}^d)\quad (\varepsilon\to 0^+)\,.\label{lem:limsup.3.3}
        \end{align}
        Moreover, exploiting \eqref{lem:limsup.3.1} and the convexity of  $(t\mapsto \smash{{t^2}})\colon\mathbb{R}\to \mathbb{R}$, for every $\lambda\in (0,1)$,~we~obtain
        \begin{align}\label{lem:limsup.4}
            \begin{aligned} 
            \smash{E}_\varepsilon^{\mathtt{d}}(v_\varepsilon)&= \tfrac{1}{2}\|\nabla v\|_{\Omega}^2-(f,v)_{\Omega}-\langle g,v\rangle_{\smash{H^{\smash{\frac{1}{2}}}(\Gamma_N)}}+\tfrac{\varepsilon}{2}\|\nabla 
            v_\varepsilon\|_{\Sigma_{I}^{\varepsilon}}^2 
            \\&\leq 
             \tfrac{1}{2}\|\nabla v\|_{\Omega}^2-(f,v)_{\Omega}-\langle g,v\rangle_{\smash{H^{\smash{\frac{1}{2}}}(\Gamma_N)}}+\smash{\tfrac{\varepsilon \lambda}{2}\|\tfrac{1}{\lambda}\overline{v}\nabla \varphi_\varepsilon\|_{\Sigma_{I}^{\varepsilon}}^2 }+\smash{\tfrac{\varepsilon (1-\lambda)}{2}\|\tfrac{1}{1-\lambda}\varphi_\varepsilon \nabla\overline{v}\|_{\Sigma_{I}^{\varepsilon}}^2}  \,,
            \end{aligned}
        \end{align}
        where, for every $\lambda\hspace{-0.1em}\in\hspace{-0.1em} (0,1)$, due to Lemma \ref{lem:transversal_distance_function}\eqref{lem:transversal_distance_function.1},\eqref{lem:transversal_distance_function.2} and the~convexity~of~${(t\hspace{-0.1em}\mapsto \hspace{-0.1em}\smash{{t^2}})\colon\hspace{-0.1em}\mathbb{R}\hspace{-0.1em}\to\hspace{-0.1em} \mathbb{R}}$, we infer that 
        \begin{align}\label{lem:limsup.5}
        \begin{aligned}
          \tfrac{\lambda \varepsilon }{2}\|\tfrac{1}{\lambda}\overline{v}\nabla \varphi_\varepsilon\|_{\Sigma_{I}^{\varepsilon}}^2&=\tfrac{\varepsilon}{2\lambda}
          \|\psi_\varepsilon\tfrac{\overline{v}}{\varepsilon \mathtt{d}^2}\nabla \mathtt{d}  +(1-\psi_\varepsilon)\tfrac{-\overline{v}}{(1-\psi_\varepsilon)\varepsilon \mathtt{d}}\nabla \psi_\varepsilon\|_{\Sigma_{I}^{\varepsilon}}^2
          \\&\leq \tfrac{\varepsilon}{2 \lambda}\big\{\|\psi_\varepsilon\|_{\infty,\Sigma_{I}^{\varepsilon}}
          \|\tfrac{\smash[t]{\overline{v}}\vphantom{v} }{\varepsilon \mathtt{d}^2} \nabla \mathtt{d} \|_{\Sigma_{I}^{\varepsilon}}^2+
          \|\tfrac{1}{1-\psi_\varepsilon}\|_{\infty,\Sigma_{I}^{\varepsilon}}
         \|\tfrac{\smash[t]{\overline{v}}\vphantom{v} }{\varepsilon \mathtt{d} }\nabla \psi_\varepsilon\|_{\Sigma_{I}^{\varepsilon}}^2\big\}
         \\&\leq \tfrac{1}{2\varepsilon\lambda}\big\{
         \varepsilon 
         \|\mathtt{d}\|_{\infty,\Gamma_I}
          \|\tfrac{\overline{v} }{\mathtt{d}^2} \nabla \mathtt{d} \|_{\Sigma_{I}^{\varepsilon}}^2+
         \tfrac{1}{1-\varepsilon 
         \|\mathtt{d}\|_{\infty,\Gamma_I}}  \big\{
         \|\tfrac{\smash[t]{\overline{v}}\vphantom{v}}{(k\cdot n)\mathtt{d}}\|_{\Sigma_{I}^{\varepsilon}}+\varepsilon\|R_\varepsilon\|_{\infty,\Sigma_{I}^{\varepsilon}}\|\overline{v}\|_{\Sigma_{I}^{\varepsilon}}\big\}^2\big\}\,.
         \end{aligned}\hspace{-7.5mm}
        \end{align}
        Using \eqref{lem:limsup.4}--\eqref{lem:limsup.5} and Lemma \ref{eq:Lebesgue_boundary_limit} with $a=((k\cdot n)\mathtt{d})^{-2}$, we find that
        \begin{align}\label{lem:limsup.6}
        \begin{aligned}
            \limsup_{\smash{\varepsilon\to 0^+}}{\big\{\tfrac{\varepsilon }{2}\|\nabla v_\varepsilon\|_{\Sigma_{I}^{\varepsilon}}^2\big\}}&\leq \tfrac{1}{2\lambda}\lim_{\varepsilon\to 0^+}{\big\{\tfrac{1}{\varepsilon}\|\tfrac{\smash[t]{\overline{v}}\vphantom{v}}{(k\cdot n)\mathtt{d}}\|_{\Sigma_{I}^{\varepsilon}}^2\big\}}
            \\&=\tfrac{1}{2\lambda}\|\smash{((k\cdot n)\mathtt{d})^{\smash{-\frac{1}{2}}}}v\|_{\Gamma_I}^2
            \to \tfrac{1}{2}\|\smash{((k\cdot n)\mathtt{d})^{\smash{-\frac{1}{2}}}}v\|_{\Gamma_I}^2\quad (\lambda\to 1^-)\,.
        \end{aligned}
        \end{align}
        Eventually, using \eqref{lem:limsup.6} in \eqref{lem:limsup.4}, we conclude that 
        \begin{align*}
            \limsup_{\smash{\varepsilon\to 0^+}}{\bigl\{\smash{E}_\varepsilon^{\mathtt{d}}(v_\varepsilon)\bigr\}}\leq \smash{E}^{\mathtt{d}}(v)\,,
        \end{align*}
        which together with \eqref{lem:limsup.3.3} yields the claimed $\limsup$-estimate.  
    \end{proof}
    \par\endgroup
    \pagebreak

    \begin{remark}[(non)necessity of piece-wise flatness of $\Gamma_I$]
        \begin{itemize}[noitemsep,topsep=2pt,leftmargin=!,labelwidth=\widthof{(ii)}]
            \item[(i)] The piece-wise flatness of the insulated boundary $\Gamma_I$ should not be necessary for the validity of Theorem \ref{thm:main}. This assumption is only needed in the proof of the $\liminf$-estimate (\textit{cf}.\ Lemma \ref{lem:liminf}) and we expect that this proof can be adapted to cover the case where the insulated boundary $\Gamma_I$ is piece-wise $C^{1,1}$. 
            
             
             \item[(ii)] The piece-wise flatness of the insulated boundary $\Gamma_I$ is not restrictive, \textit{e.g.}, in numerical simulations, where typically bounded polyhedral Lipschitz domains are considered. For a numerical study of the limit problem defined via \eqref{eq:Eh} resulting from the findings of this paper, we refer to \cite{AKKInsulationNumerics}.\enlargethispage{5mm}
        \end{itemize}
    \end{remark}

	{\setlength{\bibsep}{0pt plus 0.0ex}\small
		
		\bibliographystyle{aomplain}
		\bibliography{references}
		
	}
	
\end{document}